\documentclass[12pt]{article}
\usepackage{amsmath, amsthm, amsfonts, enumerate, amssymb}
\usepackage[pagewise, modulo]{lineno}
\usepackage[arrow, matrix, curve, all, knot]{xy}
\usepackage{graphics}
\usepackage[english]{babel}

\usepackage{ifpdf} 
\ifpdf
\usepackage{epstopdf}
\usepackage{hyperref}
\else
\usepackage[hypertex]{hyperref}
\fi

\newtheorem{thm}{Theorem}
\newtheorem{lemma}[thm]{Lemma}
\newtheorem{prop}[thm]{Proposition}
\newtheorem{cor}[thm]{Corollary}
\theoremstyle{definition}
\newtheorem{rem}[thm]{Remark}
\newtheorem{Def}[thm]{Definition}
\newtheorem{conv}[thm]{Convention}
\DeclareMathOperator{\Aut}{Aut}
\DeclareMathOperator{\End}{End}
\DeclareMathOperator{\Hom}{Hom}
\DeclareMathOperator{\id}{id}
\DeclareMathOperator{\im}{im}
\DeclareMathOperator{\Inn}{Inn}
\DeclareMathOperator{\tr}{tr}
\DeclareMathOperator{\Pic}{Pic}
\newcommand{\picc}[1]{\mathcal{P}ic(\mathcal{#1})}
\newcommand{\pic}[1]{\Pic(\mathcal{#1})}

\numberwithin{equation}{section}
\numberwithin{thm}{section}

\catcode`\@=11
\newif\if@fewtab\@fewtabtrue
{\count255=\time\divide\count255 by 60
\xdef\hourmin{\number\count255}
\multiply\count255 by-60\advance\count255 by\time
\xdef\hourmin{\hourmin:\ifnum\count255<10 0\fi\the\count255}}
\def\ps@draft{\let\@mkboth\@gobbletwo
    \def\@oddfoot{\hbox to 7 cm{\tiny \versionno
       \hfil}\hskip -7cm\hfil\rm\thepage \hfil {\tiny\draftdate}}
    \def\@oddhead{}
    \def\@evenhead{}\let\@evenfoot\@oddfoot}
\def\draftdate{\number\month/\number\day/\number\year\ \ \ \hourmin }

\catcode`\@=12

\def\bb            {{B\otimesA B}}
\def\be            {\begin{equation}}
\def\bearl         {\begin{array}{l}} 
\def\bearll        {\begin{array}{ll}}
\def\bearlll       {\begin{array}{lll}}
\def\C             {\mathcal{C}}
\def\CA            {\C_{\!A}}
\def\CAA           {\C_{\!A|A}}
\def\cir           {\,{\circ}\,}
\def\congTo        {\,{\stackrel\cong\to}\,} 
\def\D             {\mathcal{D}}
\def\dim           {\mathrm{dim}}
\def\dimf          {\dim_\Bbbk}
\def\dimL          {\dim_{\mathrm l}}
\def\dimR          {\dim_{\mathrm r}}
\def\dimLR         {\dim_{\mathrm l|r}}
\def\dsty          {\displaystyle }
\def\ee            {\end{equation}}
\def\eear          {\end{array}}
\def\Eins          {\mathbf{1}}    
\def\eq            {\,{=}\,}
\newcommand\erf[1] {\mbox{(\ref{#1})}}
\def\iN            {\,{\in}\,} 
\def\longcongTo    {\,{\stackrel\cong\longrightarrow}\,} 
\newcommand\MK[4]  {{#1}\,{\stackrel{{#3},{#4}}{\longleftrightarrow}}\,{#2}}
\def\op            {\text{op}}
\def\otA           {\,{\otimesA}\,}
\def\OtA           {{\otimesA}}
\def\otB           {\,{\otimesB}\,}
\def\otimesA       {{\otimes}_{\!A}^{}}
\def\otimesB       {{\otimes}_{\!B}^{}}
\def\oti           {\,{\otimes}\,}
\def\Oti           {{\otimes}}
\def\PsiA          {\Psi_{\!A}}
\def\rmd           {\mathrm d}
\def\To            {\,{\to}\,}
\def\trL           {\tr_{\mathrm l}}
\def\trR           {\tr_{\mathrm r}}
\def\vect          {{\mathcal V}\mbox{\sl ect}}
\def\vectk         {{\mathcal V}\mbox{\sl ect}_\Bbbk}
\def\Vee           {{}^{\vee\!}}
\newcommand\bild[1]{\includegraphics{#1}}
\def\bg         {\begin{gather}}
\def\eg         {\end{gather}}
\def\bp         {\begin{picture}}
\def\ep         {\end{picture}}
\renewcommand\epsilon{\varepsilon}\usepackage{rotating}
\begin{document}

\thispagestyle{empty}
\def\thefootnote{\fnsymbol{footnote}}
\begin{flushright}
   {\sf KCL-MTH-07-18}\\
   {\sf ZMP-HH/07-13}\\
   {\sf Hamburger$\;$Beitr\"age$\;$zur$\;$Mathematik$\;$Nr.$\;$294}
             \vskip 2em
\end{flushright}
\vskip 2.0em
\begin{center}\Large 
  ON THE ROSENBERG-ZELINSKY SEQUENCE\\ IN ABELIAN MONOIDAL CATEGORIES
\end{center}\vskip 1.5em
\begin{center}
  Till Barmeier\,$^{\,a,b}$,~
  ~J\"urgen Fuchs\,$^{\,c}$,~
  ~Ingo Runkel\,$^{\,b}$,~
  ~Christoph Schweigert\,$^{\,a}$\footnote{\scriptsize 
  ~Email addresses: \\
  $~$\hspace*{2.4em}barmeier@math.uni-hamburg.de, jfuchs@fuchs.tekn.kau.se, 
  ingo.runkel@kcl.ac.uk, schweigert@math.uni-hamburg.de}
\end{center}

\begin{center}\it$^a$
  Organisationseinheit Mathematik, \ Universit\"at Hamburg\\
  Schwerpunkt Algebra und Zahlentheorie\\
  Bundesstra\ss e 55, \ D\,--\,20\,146\, Hamburg
\end{center}
\begin{center}\it$^b$
  Department of Mathematics, King's College London \\
  Strand, London WC2R 2LS, United Kingdom  
\end{center}
\begin{center}\it$^c$
  Teoretisk fysik, \ Karlstads Universitet\\
  Universitetsgatan 5, \ S\,--\,651\,88\, Karlstad
\end{center}
\vskip 1.5em
\begin{center} December 2007 \end{center}
\vskip 2em
\begin{abstract} \noindent
We consider Frobenius algebras and their bimodules in certain abelian monoidal 
categories. In particular we study the Picard group of the category of bimodules 
over a Frobenius algebra, i.e.\ the group of isomorphism classes of invertible 
bimodules. The Rosenberg-Zelinsky sequence describes a homomorphism from the 
group of algebra automorphisms to the Picard group, which however is typically 
not surjective. We investigate under which conditions there exists a Morita 
equivalent Frobenius algebra for which the corresponding homomorphism is 
surjective. One motivation for our considerations is the orbifold 
construction in conformal field theory.
\end{abstract}

\setcounter{footnote}{0}
\def\thefootnote{\arabic{footnote}}
\newpage

\section{Introduction}

In the study of associative algebras it is often advantageous to collect algebras 
into a category whose morphisms are not algebra homomorphisms, but instead 
bimodules. One motivation for this is provided by the following observation. Let
$\Bbbk$ be a field and consider finite-dimensional unital associative 
$\Bbbk$-algebras. The condition on a $\Bbbk$-linear map to be an algebra morphism 
is obviously not linear. As a consequence the category of algebras and algebra 
homomorphisms has the unpleasant feature of not being additive.

On the other hand, instead of an algebra homomorphism $\varphi{:}~ A \To B$ 
one can equivalently consider the $B$-$A$-bimodule $B_\varphi$ which as a 
$\Bbbk$-vector space coincides with $B$ and whose left action is given by the 
multiplication of $B$ while the right action is application of $\varphi$ 
composed with multiplication in $B$. This is consistent with composition in 
the sense that given another algebra homomorphism $\psi{:}~ B \To C$ there is an 
isomorphism $C_\psi \,{\otimes_B}\, B_\varphi \,{\cong}\, C_{\psi\circ\varphi}$ 
of $C$-$A$-bimodules. It is then natural not to restrict one's attention to such 
special bimodules, but to allow all $B$-$A$-bimodules as morphisms from $A$ to 
$B$ \cite[sect.\,5.7]{bena}.
Of course, as bimodules come with their own morphisms, one then actually deals with 
the structure of a bicategory. The advantage is that the $1$-morphism category 
$A \To B$, i.e.\ the category of $B$-$A$-bimodules, is additive and even abelian. 

Taking bimodules as morphisms has further interesting consequences. First of all, 
the concept of isomorphy of two algebras $A$ and $B$ is now replaced by Morita 
equivalence, which requires the existence of an invertible $A$-$B$-bimodule. 
Indeed, in applications involving associative algebras one often finds that not only 
isomorphic but also Morita equivalent algebras can be used for a given purpose. 
The classical example is the equivalence of the category of left (or right) modules 
over Morita equivalent algebras. Another illustration is the Morita equivalence between 
invariant subalgebras and crossed products, see e.g.\ \cite{rief}. Examples in 
the realm of mathematical physics include the observations that matrix theories 
on Morita equivalent noncommutative tori are physically equivalent \cite{schw'13}, 
and that Morita equivalent symmetric special Frobenius algebras in modular tensor 
categories describe equivalent rational conformal field theories \cite{ffrs3,ffrs5}. 

As a second consequence, instead of the automorphism group $\Aut(A)$ one now 
deals with the invertible $A$-bimodules. The isomorphism classes of these particular
bimodules form the Picard group $\Pic(A\mbox{-}\mathrm{Bimod})$ of $A$-bimodules.
While Morita equivalent algebras may have different automorphism groups, the 
corresponding Picard groups are isomorphic. One finds that for any algebra $A$ the 
groups $\Aut(A)$ and $\Pic(A\mbox{-}\mathrm{Bimod})$ are related by the exact sequence
  \be
  0 \,\longrightarrow\, \Inn (A) \,\longrightarrow\, \Aut(A)
  \,\stackrel{\Psi_{\!A}}{\longrightarrow}\, \Pic(A\mbox{-}\mathrm{Bimod}) \,,
  \label{RosZe}
  \ee
which is a variant of the Rosenberg-Zelinsky \cite{roZe,KNoj} sequence. Here 
$\Inn(A)$ denotes the inner automorphisms of $A$, and the group homomorphism 
$\Psi_{\!A}$ is given by assigning to an automorphism 
$\omega$ of $A$ the bimodule $A_\omega$ obtained
from $A$ by twisting the right action of $A$ on itself by $\omega$.
In other words, $\Pic(A\mbox{-}\mathrm{Bimod})$ is the home for the 
obstruction to a Skolem-Noether theorem.

It should be noticed that the group homomorphism $\Psi_{\!A}$ in 
\erf{RosZe} is not necessarily a surjection. But for practical purposes
in concrete applications it can be of interest to have an explicit realisation
of the Picard group in terms of automorphisms of the algebra available. 
This leads naturally to the following questions:
\def\leftmargini{1.1em}
\begin{itemize}
\item 
Does there exist another algebra $A'$, Morita equivalent to $A$, such that
the group homomorphism 
$\Psi_{\!A'}{:}~ \Aut(A') \To \Pic(A'\mbox{-}\mathrm{Bimod})$ 
in \erf{RosZe} is surjective?
\item 
And, once such an algebra $A'$ has been constructed: Does
this surjection admit a section, 
i.e.\ can the group $\Pic(A\mbox{-}\mathrm{Bimod})$ be identified with a 
subgroup of the automorphism group of the Morita equivalent algebra $A'$?
\end{itemize}

We will investigate these questions in a more general setting, namely
we consider algebras in $\Bbbk$-linear monoidal categories more general than
the one of $\Bbbk$-vector spaces. Like many other
results valid for vector spaces,
also the sequence \erf{RosZe} continues to hold in this setting, see
\cite[prop.\,3.14]{vazh} and \cite[prop.\,7]{fuRs11}. 

We start in section \ref{sec:alg-cat} by collecting some aspects of algebras and 
Morita equivalence in monoidal categories and review the definition of invertible 
objects and of the Picard category. Section \ref{sec:fix-alg} collects information 
about fixed algebras under some subgroup of algebra automorphisms. In section 
\ref{sec:1-class} we answer the questions raised above for the special case that 
the algebra $A$ is the tensor unit of the monoidal category $\D$ under 
consideration. As recalled in section \ref{sec:alg-cat}, the categorical dimension 
provides a character on the Picard group with values in  $\Bbbk^\times$. The main 
result of section \ref{sec:1-class}, Proposition \ref{mengenschnitt}, supplies, 
for any finite subgroup $H$ of the Picard group on which this character is 
trivial, an algebra $A'$ that is Morita equivalent to the tensor unit such that 
the elements of $H$ can be identified with automorphisms of $A$. Theorem 
\ref{gruppenschnitt}, in turn, gives a characterisation of group homomorphisms
$H \To \Aut(A)$ in terms of cochains on $H$. In this case the subgroup $H$ is not 
only required to have trivial character, but in addition a three-cocycle on 
$\Pic(A\mbox{-}\mathrm{Bimod})$ must be trivial when restricted to $H$. The 
relevant three-cocycle is obtained from the associativity constraint of 
$\mathcal{D}$, see eq.\ \erf{defkozykel} below. We also 
compute the fixed algebra under the corresponding subgroup of automorphisms.
In section \ref{sec:gen-class} these results are generalised to algebras
not necessarily Morita equivalent to the tensor unit, providing an affirmative 
answer to the above questions also in the general case. However, similar to the 
$A\eq\Eins$ case, one needs to restrict oneself to a finite subgroup $H$ of
$\Pic(A\mbox{-}\mathrm{Bimod})$ such that the corresponding invertible bimodules
have categorical dimension equal to $1$ in $A\mbox{-}\mathrm{Bimod}$ and for 
which the associativity constraint of $A\mbox{-}\mathrm{Bimod}$ is trivial. This 
is stated in Theorem \ref{thm:main}, which is the main result of this paper.

\medskip

Let us also briefly mention a motivation of our considerations which comes from 
conformal field theory. A consistent rational conformal field theory (on oriented 
surfaces with possibly non-empty boundary) 
is determined by a module category $\mathcal{M}$ over a modular tensor category $\C$
\cite{TFT1}. The Picard group of the category of module endofunctors of
$\mathcal{M}$ describes the symmetries of this CFT \cite{ffrs5}. The explicit
construction of this CFT requires not just the abstract module
category, but rather a concrete realization as category of modules over a Frobenius
algebra $A$, as this provides a natural forgetful functor from $\mathcal{M}$
to $\C$ which enters crucially in the construction. The module endofunctors
are realised as the category of $A$-$A$-bimodules. For practical purposes it 
can be useful to choose the algebra $A$ such that a given subgroup $H$ of the 
symmetries $\Pic(A\mbox{-}\mathrm{Bimod})$ of the CFT is realised 
as automorphisms of $A$. Theorem \ref{thm:main} provides us with conditions 
for when such a representative exists. Finally, the fixed algebra under this 
subgroup of automorphisms is related to the CFT obtained by `orbifolding' 
the original CFT by the symmetry $H$.

\bigskip

\noindent{\sc Acknowledgements:}
TB is supported by the European Superstring Theory Network (MCFH-2004-512194) 
and thanks King's College London for hospitality.
JF is partially supported by VR under project no.\ 621-2006-3343.
IR is partially supported by the EPSRC First Grant EP/E005047/1, the PPARC 
rolling grant PP/C507145/1 and the Marie Curie network `Superstring Theory' 
(MRTN-CT-2004-512194). 
CS is partially supported by the Collaborative Research Centre 676 ``Particles, 
Strings and the Early Universe - the Structure of Matter and Space-Time''.


\section{Algebras in monoidal categories}\label{sec:alg-cat}

In this section we collect information about a few basic structures that 
will be needed below. Let $\D$ be an abelian category enriched over the category 
$\vectk$ of finite-dimensional vector spaces over a field $\Bbbk$.
An object $X$ of $\D$ is called simple iff it has no proper subobjects.
An endomorphism of a simple object $X$ is either zero or an isomorphism 
(Schur's lemma), and hence the endomorphism space $\Hom(X,X)$ is a finite-dimensional 
division algebra over $\Bbbk$. An object $X$ of $\D$ is called absolutely simple 
iff $\Hom (X,X) \eq \Bbbk \id_X$. If $\Bbbk$ is algebraically closed, then every 
simple object is absolutely simple; the converse holds e.g.\ if $\D$ is semisimple.

When $\D$ is monoidal, then without loss of generality we assume it to be strict. 
More specifically, for the rest of this paper we make the following assumption.

\begin{conv}\label{convention1}
$(\mathcal{D}, \otimes , \Eins)$ is an abelian strict monoidal category with simple 
and absolutely simple tensor unit $\Eins$, and enriched over $\vectk$ for a
field $\Bbbk$ of characteristic zero.
\end{conv}

\noindent
In particular, $\Hom(\Eins,\Eins) \eq \Bbbk\id_\Eins$, which we identify with $\Bbbk$.

\begin{Def}
A {\em right duality\/} on $\D$ assigns to each object $X$ of $\D$ an 
object $X^\vee$, called the right dual object of $X$, and morphisms
$b_X\iN\Hom (\Eins, X\oti X^\vee)$ and $d_X\iN\Hom (X^\vee\oti X, \Eins )$ such that
  \be
  (\id_X\oti d_X)\cir (b_X\oti \id_X) =\id_X
  \qquad{\rm and}\qquad
  (d_X\oti\id_{X^\vee})\cir (\id_{X^\vee}\oti  b_X) =\id_{X^\vee} \,.
  \ee
A {\em left duality\/} on $\D$ assigns to each object $X$ of $\D$ a left dual 
object $\Vee X$ together with morphisms $\tilde b_X\iN \Hom (\Eins, \Vee X\oti X)$ and 
$\tilde d_X\iN \Hom (X\oti \Vee X, \Eins )$ such that
  \be
  (\tilde d_X\oti \id_X)\cir (\id_X\oti \tilde b_X) =\id_X
  \qquad{\rm and}\qquad
  (\id_{{}^{\vee\!} X}{\otimes}\, \tilde d_X)\cir (\tilde b_X\oti \id_{\Vee X} )
  =\id_{\Vee X} \,.
  \ee
\end{Def}

Note that $\Eins^\vee \,{\overset{\cong}{\rightarrow}}\, \Eins^\vee \oti 
\Eins \,{\overset{d_\Eins}{\rightarrow}}\, \Eins$ is nonzero; since by assumption
$\Eins$ is simple, we thus have $\Eins^\vee \,{\cong}\, \Eins$. In the same way one 
sees that $\Vee \Eins \,{\cong}\, \Eins$. Further, given a right duality, the right 
dual morphism to a morphism $f\iN\Hom (X,Y)$ is the morphism
  \be
  f^\vee :=(d_Y\oti \id_{X^\vee})\cir (\id_{Y^\vee}\Oti\, f\oti \id_{X^\vee})\cir 
  (\id_{Y^\vee}\Oti\, b_X) ~\iN \Hom (Y^\vee, X^\vee) \,.
  \ee
Left dual morphisms are defined analogously.
Hereby each duality furnishes a functor from $\D$ to $\D^\op$. 
Further, the objects $(X\oti Y)^\vee$ and $Y^\vee \oti X^\vee$ are isomorphic.

\begin{Def}
A {\em sovereign\/}\,%
   \footnote{~What we call sovereign is sometimes referred to as {\em strictly sovereign\/},
   compare \cite{bich2,brug6}.}
category is a monoidal category that is equipped with 
a left and a right duality which coincide as functors, i.e.\ $X^\vee \eq \Vee X$ 
for every object $X$ and $f^\vee \eq \Vee f$ for every morphism $f$.
\end{Def}

In a sovereign category the {\em left\/} and {\em right traces\/} of an endomorphism 
$f\iN \Hom (X,X)$ are the scalars (remember that we identify $\End(\Eins)$ with $\Bbbk$)
  \be
  \trL (f):=d_X\cir (\id_{X^\vee}\Oti\, f)\cir \tilde b_X
  \qquad{\rm and}\qquad
  \trR (f):=\tilde d_X\cir (f\oti\id_{X^\vee}) \cir b_X \,,
  \ee
respectively, and the left and right {\em dimensions\/} of an object $X$ are the scalars
  \be
  \dimL (X):=\trL (\id_X)\,, \qquad \dimR (X):=\trR (\id_X) \,.
  \ee
Both traces are cyclic, and dimensions are constant on isomorphism classes,
multiplicative under the tensor product and additive under direct sums. 
Further, one has $\trL (f) \eq \trR (f^\vee)$, and using the fact that in a 
sovereign category each object $X$ is isomorphic to its double dual $X^{\vee\vee}$
it follows that
the right dimension of the dual object equals the left dimension of the object itself,
  \be\label{diml-dimr}
  \dimL (X) = \dimR (X^\vee) \,,
  \ee
and vice versa. In particular, any object that is isomorphic to its dual, 
$X\,{\cong}\, X^\vee$, has equal left and right dimension, which we then denote by 
$\dim(X)$.
The tensor unit $\Eins$ is isomorphic to its dual and has dimension $\dim(\Eins)\eq1$.

Next we collect some information about algebra objects in monoidal categories.
Recall that a (unital, associative) \textit{algebra\/} in $\mathcal{D}$ is a triple 
$(A, m, \eta )$ consisting of an object $A$ of $\D$ and morphisms
$m\iN\Hom (A\oti A, A)$ and $\eta\iN \Hom (\Eins , A)$, such that
  \be
  m\cir (\id_A\oti m)=m\cir (m\oti \id_A)
  \qquad{\rm and}\qquad
  m\cir (\id_A\oti \eta )=\id_A=m\cir (\eta\oti \id_A) \,.
  \ee
Dually, a (counital, coassociative) \textit{coalgebra\/} is a triple 
$(C, \Delta, \epsilon)$ with $C$ an object of $\D$ and morphisms 
$\Delta\iN \Hom(C, C\oti C)$ and $\epsilon \iN \Hom (C, \Eins)$, such that
  \be
  (\Delta \oti\id_C)\cir \Delta =(\id_C\oti\Delta)\cir\Delta
  \qquad{\rm and}\qquad
  (\id_C\oti \epsilon )\cir\Delta =\id_C=(\epsilon\oti \id_C)\cir \Delta \,.
  \ee
The following concepts are also well known, see e.g.\ \cite{muge8,TFT1}.

\begin{Def}\label{algebrendef}
~\\[-1.6em]
\def\leftmargini{2.1em}
\begin{itemize}
\item[(i)]
A \textit{Frobenius algebra\/} in $\D$ is a quintuple $(A, m, \eta, \Delta, \epsilon)$, 
such that $(A, m, \eta )$ is an algebra in $\D$, $(A, \Delta, \epsilon)$ is a 
coalgebra and the compatibility relation
  \be \label{frobeigenschaft}
  (\id_A\oti m)\cir (\Delta\oti \id_A)=\Delta\cir m=(m\oti \id_A)\cir (\id_A\oti \Delta)
  \ee
between the algebra and coalgebra structures is satisfied.
\item[(ii)] 
A Frobenius algebra $A$ is called \textit{special\/} iff $m\cir \Delta \eq \beta_A\id_A$ and 
$\epsilon\cir \eta \eq \beta_\Eins \id_\Eins $ with $\beta_\Eins, \beta_A\iN \Bbbk^\times$. 
$A$ is called \textit{normalised special\/} iff $A$ is special with $\beta_A \eq 1$.
\item[(iii)]
\label{defsymm} 
If $\D$ is in addition sovereign, an algebra $A$ in $\D$ is
called \textit{symmetric\/} iff the two morphisms
  \be\label{Phi1Phi2}
  \Phi_1:=((\epsilon\cir m)\oti \id_{A^\vee})\cir (\id_A\oti b_A)
  \quad~{\rm and}\quad~
  \Phi_2:=(\id_{A^\vee}\Oti\, (\epsilon\cir m))\cir (\tilde b_A\oti \id_A)
  \ee
in $\Hom (A, A^\vee )$ are equal.
\end{itemize}
\end{Def}

For $(A, m_A, \eta_A)$ and $(B, m_B, \eta_B)$ algebras in $\D$, a morphism $f{:}~ A\To B$
is called a (unital) morphism of algebras iff $f\cir m_A \eq m_B\cir (f\oti f)$ 
in $\Hom (A\oti A, B)$ and $f\cir \eta_A \eq \eta_B$. Similarly one defines (counital)
morphisms of coalgebras and morphisms of Frobenius algebras.
An algebra $S$ is called a \textit{subalgebra\/} of $A$ iff there is a monic 
$i{:}~ S\To A$ that is a morphism of algebras.

A (unital) {\em left\/} $A$-{\em module\/} is a pair $(M, \rho)$ consisting of an object 
$M$ in $\D$ and a morphism $\rho\iN \Hom (A\oti M, M)$, such that
  \be
  \rho\cir (\id_A\oti\rho) = \rho\cir (m\oti\id_M) \qquad{\rm and}\qquad
  \rho\cir (\eta\oti \id_M) = \id_M .
  \ee
Similarly one defines right $A$-modules. An $A$-$A$-{\em bimodule\/} 
(or $A$-{\em bimodule\/} for short) is a triple 
$(M, \rho,\varrho)$ such that $(M, \rho)$ is a left $A$-module, $(M, \varrho)$ a
right $A$-module, and the left and right actions of $A$ on $M$ commute.
Analogously, $A$-$B$-bimodules carry a left action of the algebra $A$ and a
commuting right action of the algebra $B$.

For $(M,\rho_M)$ and $(N,\rho_N)$ left $A$-modules, a morphism $f\iN\Hom (M,N)$ 
is said to be a morphism of left $A$-modules (or briefly, a module morphism) iff
$f\cir \rho_M \eq \rho_N\cir (\id_A\oti f)$. Analogously one defines morphisms of
$A$-$B$-bimodules. Thereby one obtains a category, with objects the $A$-$B$-bi\-modules 
and morphisms the $A$-$B$-bimodule morphisms. We denote this category by 
$\D_{\!A|B}$ and the set of bimodule morphisms from $M$ to $N$ by $\Hom_{A|B}(M,N)$. 
The Frobenius property \erf{frobeigenschaft} means that the coproduct $\Delta$ 
is a morphism of $A$-bimodules.

\begin{Def}
An algebra is called \textit{(absolutely) simple} iff it is 
(absolutely) simple as a bimodule over itself. Thus
$A$ is absolutely simple iff $\Hom_{A|A}(A,A) \eq \Bbbk\id_A$.
\end{Def}

\begin{rem}
Since $\D$ is abelian, one can define a tensor product of $A$-bimodules.
This turns the bimodule category $\D_{\!A|A}$ into a monoidal category. 
For example, $\D\,{\cong}\, \D_{\!\Eins|\Eins}$ as monoidal categories.
See the appendix for more details on this and especially on tensor products 
over special Frobenius algebras.
\end{rem}

\begin{rem}\label{frobdim}
If $A$ is a (not necessarily symmetric) Frobenius algebra in a sovereign 
category, then the morphisms $\Phi_1$ and $\Phi_2$ in \erf{Phi1Phi2}
are invertible, with inverses
  \be
  \Phi_1^{-1}=(d_A\oti \id_A)\cir (\id_{A^\vee}\Oti\, (\Delta\cir\eta))
  \qquad{\rm and}\qquad
  \Phi_2^{-1}=(\id_A\oti \tilde d_A)\cir ((\Delta\cir\eta )\oti\id_{A^\vee}) \,,
  \ee
respectively.
So if $A$ is Frobenius, $A$ and $A^\vee$ are isomorphic, hence the left and
right dimension of $A$ are equal. Accordingly we will write $\dim (A)$ for 
the dimension of a Frobenius algebra in the sequel.
\\[3pt]
Further one can show (see \cite{TFT1}, section 3) that for any symmetric 
special Frobenius algebra $A$ the relation $\beta_A\,\beta_\Eins \eq \dim(A)$ holds.
In particular, $\dim(A) \,{\neq}\, 0$. Furthermore, without loss of generality one 
can assume that the coproduct is normalised such that $\beta_\Eins \eq \dim (A)$ 
and $\beta_A \eq 1$, i.e.\ $A$ is normalised special. 
\end{rem}

\begin{lemma}
Let $(A, m, \eta)$ be an algebra with $\dimf\Hom (\Eins , A) \eq 1$. 
Then $A$ is an absolutely simple algebra.
\end{lemma}

\begin{proof}
By Proposition 4.7 of \cite{fuSc16} one has $\Hom(\Eins,A) \,{\cong}\, \Hom_A(A,A)$. The 
result thus follows from $1 \,{\le}\, \dimf\Hom_{A|A}(A,A) \,{\le}\, \dimf \Hom_A(A,A)$. 
\end{proof}

\begin{rem}\label{algbeispiel}
Obviously the tensor unit $\Eins$ is a symmetric special Frobenius algebra.
One also easily verifies that for any object $X$ in a sovereign category 
the object $X\oti X^\vee$ with structural morphisms
  \be
  m:=\id_X\oti d_X\oti\id_{X^\vee} \,,\quad~ \eta:=b_X \,,\quad~
  \Delta:=\id_X\oti \tilde b_X \oti \id_{X^\vee} \,, \quad~ \epsilon:=\tilde d_X
  \ee
provides an example of a symmetric Frobenius algebra. If the object $X$ has nonzero 
left and right dimensions, then this algebra is also special, with
  \be
  \beta_{X\oti X^\vee}=\dimL (X) \,,\qquad \beta_\Eins =\dimR (X) \,.
  \ee
The object $X$ is naturally a left module over $X\oti X^\vee$, with representation
morphism $\rho \eq \id_X\oti d_X$, while the object
$X^\vee$ is a right module over $X\oti X^\vee$ with $\varrho \eq d_X\oti \id_{X^\vee}$.
\end{rem}

Next we recall the concept of Morita equivalence of algebras (for details see 
e.g.\ \cite{Par,vazh}).

\begin{Def}
A {\em Morita context\/} in $\D$ is a sextuple $(A, B, P, Q, f, g)$, where $A$ 
and $B$ are algebras in $\D$, $P \,{\equiv}\, {}_AP_B$ is an $A$-$B$-bimodule 
and $Q \,{\equiv}\, {}_BQ_A$ is a $B$-$A$-bimodule, such that $f{:}~ 
P\otB Q\congTo A$ and $g{:}~ Q\otA P\congTo B$ are isomorphisms of $A$- and 
$B$-bimodules, respectively, and the two diagrams
  \be\label{moritadiagramme}
\raisebox{3.5em}{
\xymatrix@C=3em{
(P\otimesB Q) \otimesA P\ar[r]^/.7em/{f\oti\id}\ar[d]_{\cong} & A\otimesA P\ar[dd]^\cong 
\\
P\otimesB (Q\otimesA P)\ar[d]_{\id\oti g} &
\\
P\otimesB B\ar[r]_\cong &P
}
\qquad \qquad
\xymatrix@C=3em{
(Q\otimesA P)\otimesB Q\ar[r]^/.7em/{g\oti\id}\ar[d]_{\cong} & B\otimesB Q\ar[dd]^\cong 
\\
Q\otimesA (P\otimesB Q)\ar[d]_{\id\oti f} &
\\
Q\otimesA A\ar[r]_\cong &Q
}}
  \ee
commute. 
\\
If such a Morita context exists, we call the algebras $A$ and $B$ Morita equivalent. 
In the sequel we will suppress the isomorphisms $f$ and $g$ and 
write a Morita context as $\MK ABPQ$.
\end{Def}

\begin{lemma}\label{moritacontext}
Let $\D$ be in addition sovereign and let $U$ be an object of $\D$ with nonzero left 
and right dimension. Then the symmetric special Frobenius algebra $U\oti U^\vee$ 
is Morita equivalent to the tensor unit, with Morita context 
$\MK \Eins{U\oti U^\vee}{U^\vee}{U}$. 
\end{lemma}

\begin{proof}
We only need to show that $U^\vee{\otimes_{U\otimes U^\vee}}\,U \,{\cong}\, \Eins$. 
Since $U\oti U^\vee$ is symmetric special Frobenius, the idempotent $P_{U^\vee,U}$ for 
the tensor product over $U\oti U^\vee$, as described in appendix \ref{appA},
is well defined. One calculates that $P_{U^\vee, U} \eq {(\dimL(U))}^{-1}\,
\tilde b_U \cir d_U$. This implies that the tensor unit is indeed 
isomorphic to the image of $P_{U^\vee,U}$.
Finally, commutativity of the diagrams \erf{moritadiagramme} follows using 
the techniques of projectors as presented in the appendix.
\end{proof}

\begin{Def}
An object $X$ in a monoidal category is called \textit{invertible\/} iff there exists 
an object $X^\prime$ such that $X\oti X^\prime \cong \Eins \cong X^\prime\oti X$. 
\end{Def}

If the category $\D$ has small skeleton, 
then the set of isomorphism classes of invertible objects forms a group under the
tensor product. This group is called the {\em Picard group\/} $\pic{D}$ of $\D$.

\begin{lemma}\label{theabovelemma}
Let $\D$ be in addition sovereign.
\def\leftmargini{2.1em}
\begin{itemize}
\item[\rm (i)] Every invertible object of $\D$ is simple.
\item[\rm (ii)] An object $X$ in $\D$ is invertible iff $X^\vee$ is invertible.
\item[\rm (iii)] An object $X$ in $\D$ is invertible iff the morphisms
$b_X$ and $\tilde b_X$ are invertible.
\item[\rm (iv)] Every invertible object of $\D$ is absolutely simple.
\end{itemize}
\end{lemma}

\begin{proof}
(i)\, Let $X\oti X^\prime\cong X^\prime\oti X\cong \Eins$. Assume that $e{:}~ U\To X$ 
is monic for some object $U$. Then $\id_{X^\prime}\Oti\, e{:}~ X^\prime\oti U\To X^\prime\oti\,X$ 
is monic. Indeed, if $(\id_{X^\prime}\Oti\, e)\cir f \eq (\id_{X^\prime}\Oti\, e)\cir g$ 
for some morphisms $f$ and $g$, then by applying the duality morphism $d_{X^\prime}$ 
we obtain $e\cir (d_{X^\prime}\oti \id_U)\cir (\id_{{X^\prime}^\vee}\!\oti f) 
\eq e\cir (d_{X^\prime} \otimes \id_U)\cir (\id_{{X^\prime}^\vee}\!\oti g)$. As $e$ 
is monic this amounts to $(d_{X^\prime}\oti \id_U)\cir (\id_{{X^\prime}^\vee}\Oti\, f) 
\eq (d_{X^\prime}\oti \id_U)\circ (\id_{{X^\prime}^\vee}\oti g)$, which by applying 
$b_{X^\prime}$ and using the duality property of $d_{X^\prime}$ and $b_{X^\prime}$ shows 
that $f \eq g$. Thus $\id_{X^\prime}\Oti\, e{:}~ X^\prime\oti U\To X^\prime\oti X\cong \Eins$ 
is monic. As $\Eins$ is required to be simple, it is thus an isomorphism.
Then $\id_X \oti \id_{X'} \oti e$ is an isomorphism as well. By assumption there
exists an isomorphism $b{:}~ \Eins \To X \oti X'$. With the help of $b$ we can write
$e \eq (b^{-1} \oti \id_X) \cir (\id_X \oti \id_{X'} \oti e) \cir (b \oti \id_U)$. Thus
$e$ is a composition of isomorphisms, and hence an isomorphism. In summary,
$e{:} U \To X$ being monic implies that $e$ is an isomorphism. Hence $X$ is simple.
\\[.4em]
(ii)\, Note that
$X^\vee\Oti\, X^{\prime\vee}\cong (X^\prime \oti X)^\vee\cong\Eins^\vee\cong\Eins$,
and similarly $X^{\prime\vee} \Oti\, X^\vee \cong \Eins$.
\\[.4em]
(iii)\, Since by part (ii) $X^\vee$ is invertible, so is $X\oti X^\vee$. By part (i), 
$X\oti X^\vee$ is therefore simple and $b_X{:}~ \Eins\To X\oti X^\vee$ 
is a nonzero morphism between simple objects. By Schur's lemma it is an isomorphism. 
The argument for $\tilde b_X$ proceeds along the same lines, and the converse statement
follows by definition.
\\[.4em]
(iv)\, The duality morphisms give an isomorphism $\Hom(X,X) \,{\cong}\, 
\Hom(X \oti X^\vee,\Eins)$. From parts (i) and (ii) we know that 
$X \oti X^\vee \,{\cong}\, \Eins$, and so $\Hom(X,X) \,{\cong}\, \Hom(\Eins,\Eins)$. 
That $X$ is absolutely simple now follows because $\Eins$ is absolutely simple by assumption.
\end{proof}

Lemma \ref{theabovelemma} implies that for an invertible object $X$ one has
  \be \label{dimLdimR=1}
  \dimL(X)\,\dimR(X) = \dimL(X)\,\dimL(X^\vee)
  = \dimL(X\Oti X^\vee) = \dimL(\Eins) = 1 \,.
  \ee
With the help of this equality one checks that the inverse of $b_X$ is given by 
$\dimL (X)\,\tilde d_X$,
  \be 
  \dimL (X)\, \tilde d_X\cir b_X =\dimL (X)\, \dimR (X)\, \id_\Eins =\id_\Eins .
  \ee
Analogously we have $\dimR (X)\,d_X\cir \tilde b_X \eq \id_\Eins$; thus in particular
the left and right dimensions of an invertible object $X$ are nonzero. Further we have
  \be
  \dimL (X)~ b_X\cir \tilde d_X =\id_{X\otimes X^\vee} \qquad{\rm and}\qquad
  \dimR (X)~ \tilde b_X\cir d_X =\id_{X^\vee\otimes X} \,.
  \label{invertmorph} \ee

We denote the object representing an isomorphism class $g$ in $\Pic(\D)$ by $L_g$,
i.e.\ $[L_g] \eq g \iN \Pic(\D)$. Then $L_g\oti L_h\,{\cong}\, L_{gh}$. As the 
representative of the unit class $1$ we take the tensor unit, $L_1 \eq \Eins$.

\begin{lemma}\label{charakter}
Let $\D$ be in addition sovereign and $H$ a subgroup of $\,\Pic(\D)$.
\def\leftmargini{1.8em}
\begin{itemize}
\item[\rm (i)] The mappings $\,h\mapsto \dimL(L_h)$ and 
$\,h\mapsto\dimR(L_h)$ are characters on $H$.
\item[\rm (ii)] 
If $H$ is finite, then $\,\dimLR (\bigoplus_{h\in H}L_h)$ is either $0$ or $|H|$. 
It is equal to $|H|$ iff $\,\dimLR(L_h) \eq 1$ for all $h\iN H$.
\end{itemize}
\end{lemma}

\begin{proof}
Claim (i) follows directly from the multiplicativity of the left and right 
dimension under the tensor product and from the fact that the dimension only 
depends on the isomorphism class of an object. 
\\
Because of $\dimLR(\bigoplus_{h\in H}L_h) \eq \sum_{h\in H}\dimLR(L_h)$,
part (ii) is a consequence of the orthogonality of characters. 
\end{proof}

\begin{Def}
The \textit{Picard category\/} $\picc{D}$ of $\D$ is the full subcategory of 
$\D$ whose objects are direct sums of invertible objects of $\D$.
\end{Def}


\section{Fixed algebras}\label{sec:fix-alg}

We introduce the notion of a fixed algebra under a group 
of algebra automorphisms and establish some basic results on fixed algebras.

\begin{Def}\label{fixed-alg}
Let $(A, m, \eta )$ be an algebra in $\mathcal{D}$ and $H\,{\leq}\, \Aut (A)$
a group of (unital) automorphisms of $A$. Then a {\em fixed algebra\/} under the 
action of $H$ is a pair $(A^H, j)$, where $A^H$ is an object of $\D$ and 
$j {:}~ A^H\To A$ is a monic with $\alpha\cir j \eq j$
for all $\alpha\iN H$, such that the following universal property is fulfilled:
For every object $B$ in $\mathcal{D}$ and morphism
$f {:}~B\To A$ with $\alpha\cir f \eq f$ for all $\alpha\iN H$, there is a 
unique morphism $\bar f {:}~B\To A^H$ such that $j\cir \bar f \eq f$.
\end{Def}

The object $A^H$ defined this way is unique up to isomorphism. The following 
result justifies using the term `fixed algebra', rather than `fixed object'.

\begin{lemma}\label{fixed algebra}
Given $A$, $H$ and $(A^H,j)$ as in definition \ref{fixed-alg},
there exists a unique algebra structure on the 
object $A^H$ such that the inclusion $j {:}~ A^H\To A$ is a morphism of algebras.
\end{lemma}

\begin{proof}
For arbitrary $\alpha\iN H$ consider the diagrams
  \be
  \bearl ~\\[-1.3em]
\xymatrix{
&& A^H \ar[r]^{j}&A\ar[r]^\alpha &A&&A^H \ar[r]^{j}&A\ar[r]^\alpha &A\\
&& &A^H\otimes A^H \ar[u]_{m\cir (j\otimes j)}
&&\raisebox{4.4em}{\mbox{and}}&&\Eins\ar[u]_\eta }
  \qquad\qquad\\[-3.2em]~
  \eear \ee
Since $\alpha$ is a morphism of algebras, we have $\alpha\cir m\cir (j\oti j) 
\eq m\cir ((\alpha\cir j)\oti(\alpha\cir j)) \eq m\cir (j\oti j)$; as this holds 
for all $\alpha\iN H$, the universal property of the fixed algebra yields a 
unique product morphism $\mu {:}~ A^H\oti A^H\To A^H$ such that 
$j\cir \mu \eq m\cir (j\oti j)$. By associativity of $A$ ~we have 
  \be \bearll
  j\cir\mu\cir (\mu\oti\id_{A^H}) \!\!&
  = m\cir (m\oti\id_A)\cir (j\oti j\oti j)
  \\{}\\[-.9em]&
  = m\cir (\id_A\oti m)\cir (j\oti j\oti j)
  = j\cir \mu\cir (\id_{A^H}\oti\mu ) \,.
  \eear \ee
Since $j$ is monic, this implies associativity of the product morphism 
$\mu$. Similarly, applying the universal property of $(A^H,j)$ on $\eta$ gives 
a morphism $\eta^\prime {:}~ \Eins\To A^H$ that has the properties of a 
unit for the product $\mu$. So $(A^H , \mu, \eta^\prime )$ is an 
associative algebra with unit.
\end{proof}

We proceed to show that fixed algebras under finite groups of automorphisms 
always exist in the situation studied here.
Let $H \,{\le}\, \Aut (A)$ be a finite subgroup of the group of algebra 
automorphisms of $A$. Set 
  \be\label{fixalgebraprojektor}
  P = P_H := \frac{1}{|H|} \sum_{\alpha\in H} \alpha \,\in \End(A) \,.
  \ee
Then $P\cir P \eq \frac{1}{|H|^2}\sum_{\alpha, \beta\in H}\alpha\cir\beta
\eq \frac{1}{|H|^2}\sum_{\alpha, \beta^\prime\in H}\beta^\prime
\eq \frac{1}{|H|}\sum_{\beta^\prime\in H}\beta^\prime \eq P$, i.e.\ $P$ is an
idempotent. Analogously one shows that $\alpha \cir P \eq P$ for every $\alpha\iN H$.
Further, since $\mathcal{D}$ is abelian, we can write $P \eq e\cir r$ with 
$e$ monic and $r$ epi. Denote the image of $P\,{\equiv}\,P_H$ by $A_P$,
so that $e{:}~ A_P\To A$ and $r \cir e \eq \id_{A_P^{}}$.

\begin{lemma}\label{fixalgebraexistiert}
The pair $(A_P,e)$ satisfies the universal property of the fixed algebra.
\end{lemma}

\begin{proof}
{}From $r\cir e \eq \id_{A_P}$ we see that 
$\alpha\cir e \eq \alpha\cir e\cir r\cir e \eq \alpha\cir P\cir e=P\cir e=e$ 
for all $\alpha\iN H$. For $B$ an object of $\mathcal{D}$ and 
$f {:}~ B\To A$ a morphism with $\alpha\cir f \eq f$ for all $\alpha \iN H$,
set $\bar f \,{:=}\, r\cir f$. Then 
$e\cir \bar f \eq e\cir r\cir f \eq P\cir f \eq \frac1{|H|}\sum_{\alpha\in H}\alpha\cir f
\eq \frac1{|H|}\sum_{\alpha\in H}f \eq f$. 
Further, if $f^\prime$ is another morphism satisfying $e\cir f^\prime \eq f$,
then $e\cir f^\prime \eq e\cir\bar f$ and, since $e$ is monic, $\bar f \eq f^\prime$, 
so $\bar f$ is unique. 
Hence the object $A_P$ satisfies the universal property of the fixed algebra $A^H$.
\end{proof}

We would like to express the structural morphisms of the fixed algebra through 
$e$ and $r$. To this end we 
introduce a candidate product $m_P$ and candidate unit $\eta_P$ on $A_P$: we set
  \be\label{mP-from-er}
  m_P := r\circ m\circ (e\oti e) \qquad{\rm and}\qquad 
  \eta_P := r\circ\eta \,.
  \ee
Note that $e$ is a morphism of unital algebras:
  \be\label{ealgebramorph}
  \bearll
  e\circ m_P \!\! &
  = e\circ r\circ m\circ (e\oti e) = P\circ m\circ (e\oti e)
  = \frac{1}{|H|}\sum_{\alpha\in H}\alpha\circ m\circ (e\oti e)
  \\{}\\[-.8em]&
  = \frac{1}{|H|}\sum_{\alpha\in H}m\circ ((\alpha\cir e) \oti (\alpha\cir e))
  = \frac{1}{|H|}\sum_{\alpha\in H}m\circ (e\oti e) = m\circ (e\oti e) \,,
  \\{}\\[-.4em]
  e\circ \eta_P \!\! &
  = e\circ r\circ \eta = P\circ \eta = \frac{1}{|H|}\sum_{\alpha\in H}\alpha\circ \eta
  = \frac{1}{|H|}\sum_{\alpha\in H}\eta = \eta \,.
  \eear \ee

\begin{lemma}\label{fixalgebreniso}
The algebra $(A_P,m_P,\eta_P)$ is isomorphic to the algebra structure that $A_P$ 
inherits as a fixed algebra.
\end{lemma}

\begin{proof}
An easy calculation shows that $m_P$ is associative and $\eta_P$ is a unit 
for $m_P$; thus $(A_P,m_P,\eta_P)$ is an algebra. Moreover, since according to 
lemma \ref{fixed algebra} there is a unique algebra structure on $A^H$ such that 
the inclusion into $A$ is a morphism of algebras, it follows that $A_P$ and $A^H$ 
are isomorphic as algebras.
\end{proof}

In the following discussion the term fixed algebra will always refer to the algebra $A_P$.

\begin{lemma}\label{projektorweglassen}
With $P \eq \frac{1}{|H|} \sum_{\alpha\in H} \alpha \eq e\cir r$ as in 
\erf{fixalgebraprojektor}, we have the following equalities of morphisms:
  \be \bearll
  r\cir m\cir (e\oti P) = r\cir m\cir (e\oti \id_A) \,,\quad&
  r\cir m\cir (P\oti e) = r\cir m\cir (\id_A\oti e) \qquad{\rm and} \\{}\\[-.8em]
  P\cir m\cir (e\oti e) = m\cir (e\oti e) \,.
  \eear \ee
\end{lemma}

\begin{proof}
Indeed, making use of $r\cir \alpha \eq r$ and $\alpha\cir e \eq e$ for all $\alpha \iN H$,
we have
  \be \bearll
  r\circ m\circ (e\oti P) \!\! &\dsty
  = \frac{1}{|H|}\sum_{\alpha\in H}r\circ m\circ (e\oti \alpha)
  = \frac{1}{|H|}\sum_{\alpha\in H}r\circ\alpha\circ m\circ ((\alpha^{-1}\circ e)\oti\id_A)
  \\{}\\[-.8em]&\dsty
  = \frac{1}{|H|}\sum_{\alpha\in H}r\circ m\circ (e\oti \id_A)
  = r\circ m\circ (e\oti \id_A) \,.
  \eear \ee
The other two equalities are established analogously.
\end{proof}

\begin{rem}\label{projektorweg}
With the help of the graphical calculus for morphisms in strict monoidal 
categories (see \cite{joSt5,Kassel,MAji,BaKi}, and e.g.\
Appendix A of \cite{ffrs} for the graphical representation of the structural
morphisms of Frobenius algebras in such categories), the equalities in Lemma
\ref{projektorweglassen} can be visualised as follows:
  \be 
\raisebox{-23pt}{
\bp(53,57)
\bild{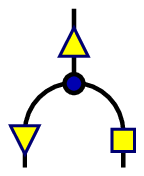}
\put(-36,-10){
\put(-1,2){\scriptsize $A_P$}
\put(28,2){\scriptsize $A$}
\put(14,61){\scriptsize $A_P$}
\put(-4.5,15){\scriptsize $e$}
\put(37,15.5){\scriptsize $P$}
\put(23,45){\scriptsize $r$}
}
\ep
} =\
\raisebox{-23pt}{
\bp(40,50)
\bild{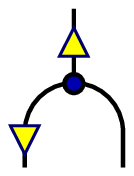}
\put(-35,-10){
\put(0,2){\scriptsize $A_P$}
\put(30,2){\scriptsize $A$}
\put(15,61){\scriptsize $A_P$}
\put(-2,15){\scriptsize $e$}
\put(25.3,45){\scriptsize $r$}
}
\ep
}\qquad\quad
\raisebox{-23pt}{
\bp(53,50)
\bild{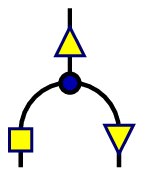}
\put(-40,-10){
\put(2,2){\scriptsize $A$}
\put(30,2){\scriptsize $A_P$}
\put(15,61){\scriptsize $A_P$}
\put(-6,16){\scriptsize $P$}
\put(39,16){\scriptsize $e$}
\put(26,45){\scriptsize $r$}
}
\ep
}=\
\raisebox{-23pt}{
\bp(40,50)
\bild{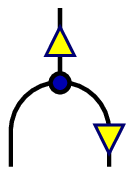}
\put(-40,-10){
\put(2,2){\scriptsize $A$}
\put(30,2){\scriptsize $A_P$}
\put(15,61){\scriptsize $A_P$}
\put(39,16){\scriptsize $e$}
\put(26.5,45){\scriptsize $r$}
}
\ep
}\qquad\quad
\raisebox{-23pt}{
\bp(53,50)
\bild{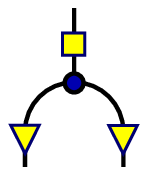}
\put(-40,-10){
\put(0,2){\scriptsize $A_P$}
\put(30,2){\scriptsize $A_P$}
\put(17,61){\scriptsize $A$}
\put(-2,16){\scriptsize $e$}
\put(39,16){\scriptsize $e$}
\put(25,44){\scriptsize $P$}
}
\ep
}=\
\raisebox{-23pt}{
\bp(60,50)
\bild{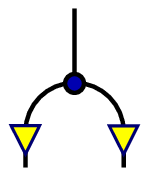}
\put(-40,-10){
\put(0,2){\scriptsize $A_P$}
\put(30,2){\scriptsize $A_P$}
\put(17,61){\scriptsize $A$}
\put(-2,16){\scriptsize $e$}
\put(39,16){\scriptsize $e$}
}
\ep
}
  \ee
~\\
If $A$ is a Frobenius algebra, it is understood that $\Aut(A)$ consists of all
algebra automorphisms of $A$ which are at the same time also coalgebra 
automorphisms. Then for a Frobenius algebra $A$ the idempotent $P$ can also be 
omitted in the following situations, which we describe again pictorially:
  \be \bearl
\raisebox{-25pt}{
\bp(53,57)
\bild{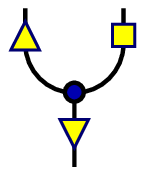}
\put(-40,-10){
\put(2,61){\scriptsize $A_P$}
\put(32,61){\scriptsize $A$}
\put(17,2){\scriptsize $A_P$}
\put(-1.5,47){\scriptsize $r$}
\put(41,46){\scriptsize $P$}
\put(26,18){\scriptsize $e$}
}
\ep
}\qquad\qquad
\raisebox{-25pt}{
\bp(53,40)
\bild{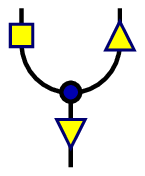}
\put(-40,-10){
\put(3,61){\scriptsize $A$}
\put(31,61){\scriptsize $A_P$}
\put(15,2){\scriptsize $A_P$}
\put(-6.3,46){\scriptsize $P$}
\put(40,47){\scriptsize $r$}
\put(25,18){\scriptsize $e$}
}
\ep
}\qquad\qquad
\raisebox{-25pt}{
\bp(40,40)
\bild{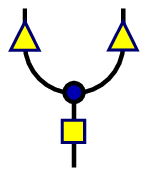}
\put(-40,-10){
\put(1,61){\scriptsize $A_P$}
\put(31,61){\scriptsize $A_P$}
\put(17,2){\scriptsize $A$}
\put(-2,47){\scriptsize $r$}
\put(40,47){\scriptsize $r$}
\put(25.7,18){\scriptsize $P$}
}
\ep
}
\\[-.6em]~
  \eear \ee
\end{rem}

\begin{prop}\label{fixalgebraeigenschaften}
Let $A$ be a Frobenius algebra in $\D$ and $H\,{\le}\, \Aut(A)$ a finite group 
of automorphisms of\ $A$.
\def\leftmargini{2.1em}
\begin{itemize}
\item[\rm (i)] 
$A_P$ is a Frobenius algebra, and the embedding $e{:}~ A_P \To A$ is
a morphism of algebras while the restriction $r{:}~ A \To A_P$ is
a morphism of coalgebras.
\item[\rm (ii)] 
If the category $\mathcal{D}$ is sovereign and $A$ is symmetric, 
then $A_P$ is symmetric, too. 
\item[\rm (iii)] 
If the category $\mathcal{D}$ is sovereign, $A$ is symmetric special and $A_P$ 
is an absolutely simple algebra and has nonzero left 
(equivalently right, cf.\ remark \ref{frobdim}) dimension, then $A_P$ is special.
\end{itemize}
\end{prop}

\begin{proof}
(i) 
The algebra structure on $A_P$ has already been defined in \erf{mP-from-er}, and 
according to \erf{ealgebramorph} $e$ is a morphism of algebras.
Denoting the coproduct on $A$ by $\Delta$ and the counit by $\varepsilon$, we 
further set $\Delta_P \,{:=}\, (r\oti r)\circ \Delta\circ e$ and 
$\varepsilon_P \,{:=}\, \varepsilon\circ e$. Similarly to the calculation in 
\erf{ealgebramorph} one verifies that $r$ is a morphism of coalgebras, and that 
$\Delta_P$ is coassociative and $\varepsilon_P$ is a counit. Regarding the 
Frobenius property, we
give a graphical proof of one of the equalities that must be satisfied:
  \begin{eqnarray}
\Delta_P \circ m_P ~=\quad
\raisebox{-40pt}{
\bp(45,93)
\bild{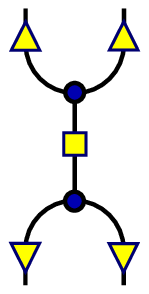}
\put(-40,-10){
\put(0,96){\scriptsize $A_P$}
\put(30,96){\scriptsize $A_P$}
\put(0,2){\scriptsize $A_P$}
\put(31,2){\scriptsize $A_P$}
\put(-3,80){\scriptsize $r$}
\put(41,80){\scriptsize $r$}
\put(26,49){\scriptsize $P$}
\put(-3,17){\scriptsize $e$}
\put(40,17){\scriptsize $e$}
}
\ep
}~~=~~
\raisebox{-40pt}{
\bp(40,90)
\bild{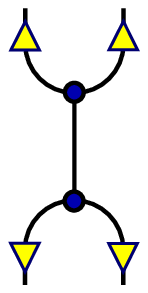}
\put(-40,-10){
\put(1,96){\scriptsize $A_P$}
\put(30,96){\scriptsize $A_P$}
\put(2,2){\scriptsize $A_P$}
\put(31,2){\scriptsize $A_P$}
\put(-3,80){\scriptsize $r$}
\put(41,80){\scriptsize $r$}
\put(-3,17){\scriptsize $e$}
\put(40,17){\scriptsize $e$}
}
\ep
}~~=~~
\raisebox{-40pt}{
\bp(75,90)
\bild{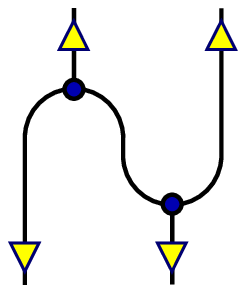}
\put(-40,-10){
\put(-13,96){\scriptsize $A_P$}
\put(30,96){\scriptsize $A_P$}
\put(-27,2){\scriptsize $A_P$}
\put(16,2){\scriptsize $A_P$}
\put(-17,80){\scriptsize $r$}
\put(41,80){\scriptsize $r$}
\put(-17,17){\scriptsize $e$}
\put(26,17){\scriptsize $e$}
}
\ep
}~~=~~
\raisebox{-40pt}{
\bp(90,90)
\bild{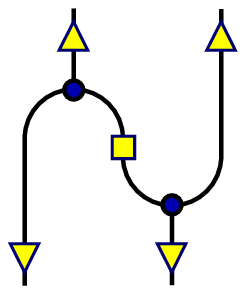}
\put(-40,-10){
\put(-13,96){\scriptsize $A_P$}
\put(30,96){\scriptsize $A_P$}
\put(-27,2){\scriptsize $A_P$}
\put(16,2){\scriptsize $A_P$}
\put(-17,80){\scriptsize $r$}
\put(42,80){\scriptsize $r$}
\put(12.5,48.5){\scriptsize $P$}
\put(-15.5,17){\scriptsize $e$}
\put(26,17){\scriptsize $e$}
}
\ep
}
\nonumber\\[1.9em]
=~ (m_P \oti \id_{A_P^{}}) \circ (\id_{A_P^{}} \Oti\, \Delta_P) \,.~~
\end{eqnarray}
Here it is used that according to remark \ref{projektorweg} we are allowed to remove
and insert idempotents $P$, and then the Frobenius property of $A$ is invoked. 
The other half of the Frobenius property is seen analogously.
\\[.5em]
(ii) The following chain of equalities shows that $A_P$ is symmetric:
  \be
\raisebox{-40pt}{
\bp(70,87)
\bild{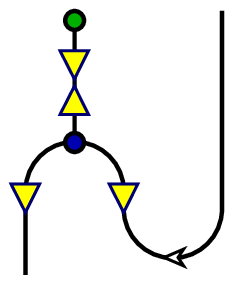}
\put(-40,-10){
\put(34,92){\scriptsize $A_P^\vee$}
\put(-25,2){\scriptsize $A_P$}
\put(-14,70){\scriptsize $e$}
\put(-14,58){\scriptsize $r$}
\put(-27,30){\scriptsize $e$}
\put(15,30){\scriptsize $e$}
}
\ep
}~~=~~
\raisebox{-40pt}{
\bp(70,80)
\bild{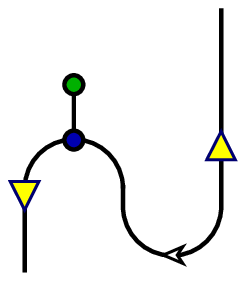}
\put(-40,-10){
\put(30,92){\scriptsize $A_P^\vee$}
\put(-28,2){\scriptsize $A_P$}
\put(-31,30){\scriptsize $e$}
\put(21.5,44){\scriptsize $e^\vee$}
}
\ep
}~~=~~
\raisebox{-40pt}{
\bp(70,80)
\bild{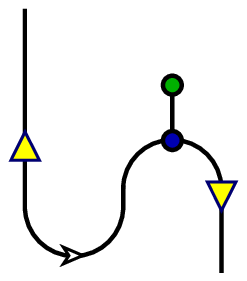}
\put(-40,-10){
\put(-27,92){\scriptsize $A_P^\vee$}
\put(29,2){\scriptsize $A_P$}
\put(-17,46){\scriptsize $e^\vee$}
\put(25,31){\scriptsize $e$}
}
\ep
}~~=~~
\raisebox{-40pt}{
\bp(70,80)
\bild{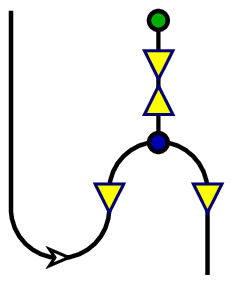}
\put(-40,-10){
\put(-27,92){\scriptsize $A_P^\vee$}
\put(29,2){\scriptsize $A_P$}
\put(26,70){\scriptsize $e$}
\put(26,58){\scriptsize $r$}
\put(-3,30){\scriptsize $e$}
\put(38,30){\scriptsize $e$}
}
\ep }  
  \ee
~\\
Here the notations
  \be
b_X~=~
\raisebox{-10pt}{
\bp(37,38)
\bild{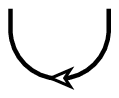}
\put(-40,-10){
\put(8,37){\scriptsize $X$}
\put(34,37){\scriptsize $X^\vee$}
}
\ep
}\qquad\mathrm{and}~~\qquad \tilde b_X~=~
\raisebox{-10pt}{
\bp(30,34)
\bild{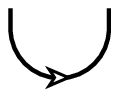}
\put(-40,-10){
\put(6,37){\scriptsize $X^\vee$}
\put(36,37){\scriptsize $X$}
}
\ep }  
  \ee
are used for the duality morphisms $b_X$ and $\tilde b_X$, respectively, of an 
object $X$. The morphisms $d_X$ and $\tilde d_X$ are drawn in a similar way. 
\\[.5em]
(iii) We have $\varepsilon_P\cir\eta_P \eq \varepsilon\cir e\cir r\cir\eta
\eq \varepsilon\cir\eta$, which is nonzero by specialness of $A$.
As $A_P$ is associative, $m_P$ is a morphism of $A_P$-bimodules. The Frobenius 
property ensures that $\Delta_P$ is also a morphism of bimodules. Hence 
$m_P\cir\Delta_P$ is a morphism of bimodules, and by absolute simplicity of $A_P$ 
it is a multiple of the identity. Moreover, $m_P\cir\Delta_P$ is not zero: we have 
  \be
  \varepsilon_P\circ m_P\circ\Delta_P\circ\eta_P
  = \varepsilon\circ m\circ (\id_A\oti P)\circ\Delta\circ\eta
  \ee
which, as $A$ is symmetric, is equal to $\trL(P) \eq \dimL(A_P)\,{\neq}\, 0$. 
We conclude that $m_P\cir\Delta_P \,{\neq}\, 0$. Hence $A_P$ is special.
\end{proof}

\begin{rem} 
In the above discussion the category $\mathcal{D}$ is assumed to be abelian, but 
this assumption can be relaxed. Of the properties of an abelian category we only
used that the morphism sets are abelian groups, that composition is bilinear, and that 
the relevant idempotents factorise in a monic and an epi, i.e.\ that $\mathcal{D}$ 
is idempotent complete. In addition we assumed
that morphisms sets are finite-dimensional $\Bbbk$-vector spaces.
\\[.2em]
{}From eq.\ \erf{fixalgebraprojektor} onwards, and in particular in proposition 
\ref{fixalgebraeigenschaften}, it is in addition used that $\D$ is 
enriched over $\vectk$. If this is not the case, one can no longer, in general, 
define an idempotent $P$ through $\frac{1}{|H|}\sum_{\alpha\in H}\alpha$, and there 
need not exist a coproduct on the fixed algebra $A^H$, even if there is one on $A$.
\end{rem}


\section{Algebras in the Morita class of the tensor unit}\label{sec:1-class}

Recall that according to our convention \ref{convention1}
$(\mathcal{D}, \otimes, \Eins )$ is abelian strict monoidal, with simple
and absolutely simple tensor unit and enriched over $\vectk$ with $\Bbbk$
of characteristic zero. {}From now on we further assume that $\mathcal{D}$ 
is skeletally small and sovereign.

We now associate to an algebra $(A, m, \eta )$ in $\D$ a specific subgroup of 
its automorphism group -- the inner automorphisms -- which are defined as follows.
The space $\Hom (\Eins, A)$ becomes a $\Bbbk$-algebra by defining the product 
as $f\,{\ast}\, g \,{:=}\, m\cir (f\oti g)$ for $f,g \iN \Hom (\Eins, A)$.
The morphism $\eta\iN \Hom (\Eins, A)$ is a unit for this product. We call 
a morphism $f$ in $\Hom (\Eins, A)$ invertible iff there exists a morphism 
$f^{-}\iN \Hom (\Eins, A)$ such that 
$f \,{\ast}\, f^{-} \eq \eta \eq f^{-} \,{\ast}\, f$. Now the morphism
  \be
  \omega_f :=m\cir (m\oti f^{-})\cir (f\oti \id_A)~\iN\Hom (A,A)
  \ee
is easily seen to be an algebra automorphism. The automorphisms of this 
form are called \textit{inner\/} automorphisms; they form a normal
subgroup $\Inn (A)\,{\le}\,\Aut (A)$ as is seen below.

\begin{Def}
For $A$ an algebra in $\D$ and $\alpha, \beta\iN \Aut (A)$, the $A$-bimodule
${}_\alpha A_\beta \eq (A,\rho_\alpha,\varrho_\beta)$ is the bimodule which 
has $A$ as underlying object and left and right actions of $A$ given by
  \be
  \rho_\alpha := m\cir (\alpha\oti\id_A) \qquad{\rm and}\qquad
  \varrho_\beta := m\cir (\id_A\oti\beta) \,,
  \ee
respectively. These left and right actions of $A$ are said to be \textit{twisted\/}
by $\alpha$ and $\beta$, respectively, and ${}_\alpha A_\beta$ is called a
\textit{twisted bimodule\/}.
\end{Def}

That this indeed defines an $A$-bimodule structure on the object $A$ is easily 
checked with the help of the multiplicativity and unitality of $\alpha$ and $\beta$.
Further, as shown in \cite{vazh,fuRs11}, the bimodules  ${}_\alpha A_\beta$ are 
invertible. Denote the isomorphism class of a bimodule $X$ by $[X]$. By setting 
  \be
  \PsiA (\alpha ) :=\, {}[{}_{\id}A_\alpha]
  \ee
one obtains an exact sequence 
  \be\label{sequenz}
  0\longrightarrow \Inn (A)\longrightarrow \Aut (A)\stackrel{\PsiA}{\longrightarrow}
  \Pic (\mathcal{D}_{A|A})
  \ee
of groups. In particular one sees that the subgroup $\Inn (A)$ is in fact a normal 
subgroup, as it is the kernel of the homomorphism $\PsiA$. The proof of exactness
of this sequence in \cite{vazh,fuRs11} is not only valid 
in braided monoidal categories, but also in the present more general situation.

Let now $A$ and $B$ be Morita equivalent algebras in $\mathcal{D}$, with
$\MK ABPQ$ a Morita context ($P \,{\equiv}\,{}_AP_B\,, Q \,{\equiv}\, {}_BQ_A$).
Then the mapping
  \be
  \begin{array}{lrcl}
  \Pi_{Q,P}:
  & \Pic (\D_{\!A|A})& \stackrel{\cong}{\longrightarrow} & \Pic (\D_{B|B})
  \\{}\\[-1.1em]
  & [X] & \longmapsto & [Q \otA X \otA P] 
  \eear \label{piciso}
  \ee
constitutes an isomorphism between the 
Picard groups $\Pic (\mathcal{D}_{A|A})$ and $\Pic (\mathcal{D}_{B|B})$.
In particular, if $A$ is an algebra that is Morita equivalent to the tensor unit
$\Eins$, then we have an isomorphism $\Pic (\mathcal{D}_{A|A})\,{\cong}\, \pic{D}$. 
As Morita equivalent algebras need not have isomorphic automorphism groups, the 
images of the group homomorphisms $\PsiA{:}~ \Aut (A)\To \Pic (\D_{\!A|A})$ and 
$\Psi_B{:}~ \Aut (B) \to \Pic (\D_{B|B})$ will in general be non-isomorphic.

In the following we will consider subgroups of the group $\Pic(\D)$. For 
a subgroup $H\le\Pic(\D)$ we put 
  \be\label{Q(H)}
  Q \,\equiv\, Q(H):=\bigoplus_{h\in H}L_h.
  \ee

\begin{rem}\label{qselbstdual}
Since the object $Q$ is the direct sum over a whole subgroup of $\pic{D}$ and 
$L_g^\vee\,{\cong}\, L_{g^{-1}}$, it follows that $Q\,{\cong}\, Q^\vee$. As a 
consequence, left and right dimensions of $Q$ are equal,
and accordingly in the sequel we use the notation $\dim(Q)$ for both of them.
\end{rem}

\begin{prop}\label{mengenschnitt}
Let $H\,{\le}\,\pic{D}$ be a finite subgroup such that 
$\,\dim(Q)\,{\ne}\,0$ for $Q \,{\equiv}\, Q(H)$. Then with the algebra
  \be\label{A(H)}
  A \,\equiv\, A(H) := Q\oti Q^\vee
  \ee
and the Morita context $\MK\Eins{A}{Q^\vee\!}{Q}$ introduced in lemma 
\ref{moritacontext}, we have $H \eq \im(\Pi_{Q^\vee\!,Q} {\circ} \PsiA)$, 
i.e.\ the subgroup $H$ is recovered as the image of the composite map
  \be
  \Aut (A) \stackrel{\Psi_{\!A}}{\longrightarrow} \Pic(\D_{\!A|A})
  \stackrel{\Pi_{Q^\vee\!,Q}}{-\!\!\!\!-\!\!\!\!-\!\!\!\!\longrightarrow} \Pic(D) \,.
  \ee
\end{prop}

\begin{proof}
The isomorphism $\Pi_{Q,Q^\vee}{:}~ \Pic (\D) \To \Pic(\D_{\!A|A})$ is given by 
$[L_g] \,{\mapsto}\, [Q\oti L_g\oti Q^\vee ]$. For $h\iN H$ we want to find 
automorphisms $\alpha_h$ of $A$ such that $_{\id}A_{\alpha_h} \,{\cong}\,
Q\oti L_h\oti Q^\vee$ as $A$-bimodules. We first observe the isomorphisms 
$Q\oti L_h\,{\cong}\,\bigoplus_{g\in H}L_g\oti L_h \cong \bigoplus_{g\in H}
L_{gh} \,{\cong}\, Q$.  We make a (in general non-canonical) choice of 
isomorphisms $f_h {:}~ Q\oti L_h \congTo Q$, with the morphism $f_1$ chosen to 
be the identity $\id_Q$. Then 
for each $h\iN H$ we define the endomorphism $\alpha_h$ of $Q\oti Q^\vee$ by
  \be \label{alphah}
\alpha_h ~:=\qquad
\raisebox{-25pt}{
\bp(70,60)
\bild{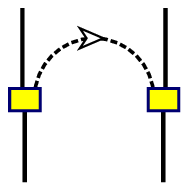}
\put(-40,-10){
\put(30,66){\scriptsize $Q^\vee$}
\put(30,2){\scriptsize $Q^\vee$}
\put(-9,66){\scriptsize $Q$}
\put(-8.5,2){\scriptsize $Q$}
\put(-26.5,32){\scriptsize $f_h^{-1}$}
\put(3.5,41){\scriptsize $L_h$}
\put(40.8,32){\scriptsize $f_h^\vee$}
}
\ep }  
  \ee
~\\
These are algebra morphisms:
\begin{eqnarray}
m \circ (\alpha_h \oti \alpha_h) ~=\quad~
        \raisebox{-30pt}{
  \begin{picture}(84,75)
   \bild{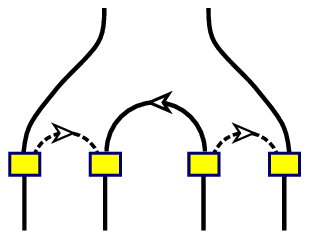}
   \put(-151,-235){
   \put(91,306){\scriptsize $Q$}
   \put(119,306){\scriptsize $Q^\vee$}
   \put(68,228){\scriptsize $Q$}
   \put(90,228){\scriptsize $Q^\vee$}
   \put(117,228){\scriptsize $Q$}
   \put(141,228){\scriptsize $Q^\vee$}
   \put(50.2,253){\scriptsize \scriptsize $f_h^{-1}$}
   \put(78.5,253){\scriptsize \scriptsize $f_h^\vee$}
   \put(103.2,253){\scriptsize \scriptsize $f_h^{-1}$}
   \put(123.5,269.5){\scriptsize $L_h$}
   \put(130.5,253){\scriptsize \scriptsize $f_h^\vee$}
     }
  \end{picture}}~~~~
=~~\raisebox{-30pt}{
        \begin{picture}(155,75)
        \bild{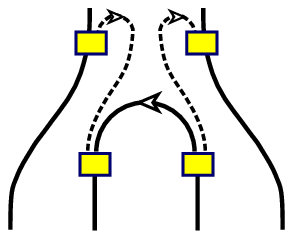}
        \put(-151,-235){
        \put(90,306){\scriptsize $Q$}
        \put(122,306){\scriptsize $Q^\vee$}
        \put(67,228){\scriptsize $Q$}
        \put(91,228){\scriptsize $Q^\vee$}
        \put(120,228){\scriptsize $Q$}
        \put(145,228){\scriptsize $Q^\vee$}
        \put(74.7,290){\scriptsize \scriptsize $f_h^{-1}$}
        \put(80,252){\scriptsize \scriptsize $f_h^\vee$}
        \put(105.2,253){\scriptsize \scriptsize $f_h^{-1}$}
        \put(132.7,289){\scriptsize \scriptsize $f_h^\vee$}
        }\end{picture}}
\nonumber\\
=~~ \raisebox{-30pt}{
        \begin{picture}(85,99)
        \bild{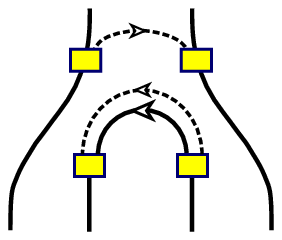}
        \put(-151,-235){
        \put(93,306){\scriptsize $Q$}
        \put(122,306){\scriptsize $Q^\vee$}
        \put(70,228){\scriptsize $Q$}
        \put(92,228){\scriptsize $Q^\vee$}
        \put(121,228){\scriptsize $Q$}
        \put(144,228){\scriptsize $Q^\vee$}
        \put(76.2,285){\scriptsize \scriptsize $f_h^{-1}$}
        \put(98.5,298){\scriptsize $L_h$}
        \put(81.2,253){\scriptsize \scriptsize $f_h^\vee$}
        \put(107.2,253){\scriptsize \scriptsize $f_h^{-1}$}
        \put(133.2,284){\scriptsize \scriptsize $f_h^\vee$}
        }\end{picture}
}
~=~~ \raisebox{-30pt}{
        \begin{picture}(84,90)
        \bild{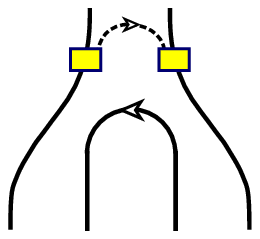}
        \put(-151,-235){
        \put(100,306){\scriptsize $Q$}
        \put(123,306){\scriptsize $Q^\vee$}
        \put(77,228){\scriptsize $Q$}
        \put(99,228){\scriptsize $Q^\vee$}
        \put(125,228){\scriptsize $Q$}
        \put(146,228){\scriptsize $Q^\vee$}
        \put(83.2,285){\scriptsize $f_h^{-1}$}
        \put(134.5,284){\scriptsize \scriptsize $f_h^\vee$}
        }\end{picture}
}=~\alpha_h\circ m\,.
  \end{eqnarray}
~\\
Here we have used that by lemma \ref{charakter}(ii) we have $\dimLR (L_h) \eq 1$ 
for $h\iN H$. The third equality is then a consequence of 
$\id_{L_h^\vee\otimes L_h} {=}\, \tilde b_{L_h}{\circ}\, d_{L_h}$, see equation 
\erf{invertmorph}; in the fourth equality $f_h$ is cancelled against $f_h^{-1}$ 
by using properties of the duality. (Also, for better readability, here and 
below we refrain from labelling some of the $L_h$-lines.)
\\[3pt]
Further, the morphisms $\alpha_h$ are also unital:
  \be
\alpha_h \circ \eta ~=\qquad
\raisebox{-30pt}{
\bp(60,66)
\bild{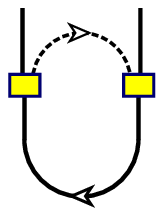}
\put(-40,-10){
\put(31,73){\scriptsize $Q^\vee$}
\put(-2.4,73){\scriptsize $Q$}
\put(-19,44){\scriptsize $f_h^{-1}$}
\put(40.5,43){\scriptsize $f_h^\vee$}
}
\ep
}~=\qquad
\raisebox{-30pt}{
\bp(45,60)
\bild{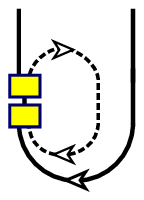}
\put(-40,-10){
\put(35,69){\scriptsize $Q^\vee$}
\put(3,69){\scriptsize $Q$}
\put(-10,43){\scriptsize $f_h^{-1}$}
\put(-7.5,28.5){\scriptsize $f_h$}
}
\ep
}~=~~
\raisebox{-30pt}{
\bp(40,60)
\bild{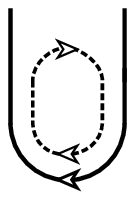}
\put(-40,-10){
\put(14,34){\scriptsize $L_h$}
\put(33,69){\scriptsize $Q^\vee$}
\put(1,69){\scriptsize $Q$}
}
\ep
}~=~~
\raisebox{-30pt}{
\bp(45,60)
\bild{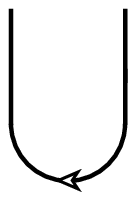}
\put(-40,-10){
\put(33,69){\scriptsize $Q^\vee$}
\put(1,69){\scriptsize $Q$}
}
\ep
} =~ \eta\,,
  \ee
where again by lemma \ref{charakter} we have $\dimR(L_h) \eq 1$. 
The inverse of $\alpha_h$ is given by
  \be
\alpha_h^{-1}~=~\quad
\raisebox{-30pt}{
\bp(73,65)
\put(0,5){
\bild{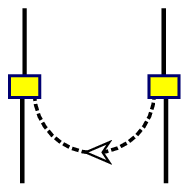}
\put(-40,-10){
\put(30,67){\scriptsize $Q^\vee$}
\put(31,2){\scriptsize $Q^\vee$}
\put(-9,67){\scriptsize $Q$}
\put(-10,2){\scriptsize $Q$}
\put(-21,37){\scriptsize $f_h^{}$}
\put(40.5,37){\scriptsize $f_h^{-\vee}$}
\put(3.6,26){\scriptsize $L_h$}
}}
\ep }
  \ee
as is seen in the following calculations:
  \begin{eqnarray}&&
\alpha_h\cir\alpha_h^{-1}~=\qquad
\raisebox{-36pt}{
\bp(71,75)
\put(0,6){ 
\bild{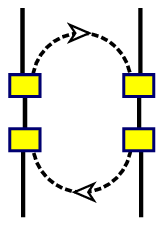}
\put(-40,-10){
\put(-2.5,76){\scriptsize $Q$}
\put(-3,2){\scriptsize $Q$}
\put(31.5,76){\scriptsize $Q^\vee$}
\put(31,2){\scriptsize $Q^\vee$}
\put(-19,48){\scriptsize $f_h^{-1}$}
\put(40,31){\scriptsize $f_h^{-\vee}$}
\put(-15,30){\scriptsize $f_h$}
\put(40,48){\scriptsize $f_h^{\vee}$}
}}
\ep
}=\quad~
\raisebox{-36pt}{
\bp(54,66)
\put(0,6){
\bild{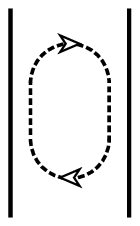}
\put(-40,-10){
\put(34,76){\scriptsize $Q^\vee$}
\put(34,2){\scriptsize $Q^\vee$}
\put(1.5,76){\scriptsize $Q$}
\put(1.7,2){\scriptsize $Q$}
\put(12,40){\scriptsize $L_h$}
}}
\ep
}=~\id_{Q\oti Q^\vee} \,,
\nonumber\\&&
\alpha_h^{-1}\cir \alpha_h~=\qquad
\raisebox{-40pt}{
\bp(70,110)
\put(0,4){
\bild{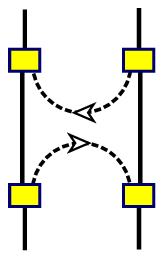}
\put(-40,-10){
\put(30,86){\scriptsize $Q^\vee$}
\put(30,2){\scriptsize $Q^\vee$}
\put(-2.5,86){\scriptsize $Q$}
\put(-2,2){\scriptsize $Q$}
\put(-18,25){\scriptsize $f_h^{-1}$}
\put(40,63){\scriptsize $f_h^{-\vee}$}
\put(-14,63){\scriptsize $f_h$}
\put(40,25){\scriptsize $f_h^{\vee}$}
}}
\ep
}=\quad~~~
\raisebox{-40pt}{
\bp(47,110)
\put(0,4){
\bild{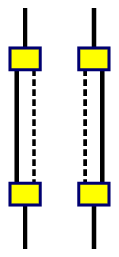}
\put(-40,-10){
\put(30,86){\scriptsize $Q^\vee$}
\put(30,2){\scriptsize $Q^\vee$}
\put(11,86){\scriptsize $Q$}
\put(11,2){\scriptsize $Q$}
\put(-6,25){\scriptsize $f_h^{-1}$}
\put(41,63){\scriptsize $f_h^{-\vee}$}
\put(-1,63){\scriptsize $f_h$}
\put(41,25){\scriptsize $f_h^{\vee}$}
}}
\ep
}
~=~\id_{Q\oti Q^\vee} \,.\quad
\end{eqnarray}
~\\
where in particular \erf{invertmorph} and $\dimLR (L_h)=1$ is used.
\\
A bimodule isomorphism $Q\oti L_h\oti Q^\vee\To {}_{\id}A_{\alpha_h}$ is now given by 
  \be
  F_h:=\, \id_Q \oti ((\tilde d_{L_h}\oti\id_{Q^\vee} )\circ
  (\id_{L_h}\Oti\, f_h^\vee )) \,.
  \ee
First we see that $F_h$ is invertible with inverse
$F_h^{-1} \eq \id_Q\oti ((\id_{L_h}\Oti\, f_h^{-\vee})\cir (b_{L_h}\oti \id_{Q^\vee}))$, 
where $f_h^{-\vee}$ stands for the dual of the inverse of $f_h$. That $F_h^{-1}$ is 
indeed inverse to $F_h$ is seen as follows:
  \begin{eqnarray}&&
F_h\cir F_h^{-1}~=\quad~
\raisebox{-23pt}{
\bp(64,64)
\bild{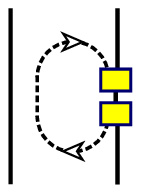}
\put(-35,-10){
\put(26,67){\scriptsize $Q^\vee$}
\put(26,2){\scriptsize $Q^\vee$}
\put(-3,67){\scriptsize $Q$}
\put(-3.5,2){\scriptsize $Q$}
\put(36.5,41){\scriptsize $f_h^\vee$}
\put(35.5,27){\scriptsize $f_h^{-\vee}$}
}
\ep
}=~~~
\raisebox{-23pt}{
\bp(50,51)
\bild{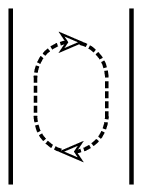}
\put(-35,-10){
\put(29.5,67){\scriptsize $Q^\vee$}
\put(30,2){\scriptsize $Q^\vee$}
\put(-4,67){\scriptsize \scriptsize $Q$}
\put(-4.4,2){\scriptsize \scriptsize $Q$}
}
\ep
}=~\id_{Q\oti Q^\vee} \,,
\nonumber\\&&
F_h^{-1}\cir F_h~=\quad
\raisebox{-35pt}{
\bp(58,93)
  \put(0,12){
\bild{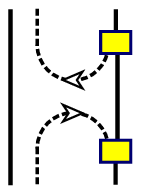}
\put(-35,-10){
\put(27,67){\scriptsize $Q^\vee$}
\put(27,2){\scriptsize $Q^\vee$}
\put(-5,67){\scriptsize $Q$}
\put(-5,2){\scriptsize $Q$}
\put(5,2){\scriptsize $L_h$}
\put(5,67){\scriptsize $L_h$}
\put(36,50.5){\scriptsize $f_h^{-\vee}$}
\put(36,18.7){\scriptsize $f_h^\vee$}
}
  }
\ep
}~~=~~~
\raisebox{-35pt}{
\bp(53,80)
  \put(0,12){
\bild{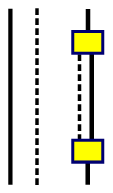}
\put(-35,-10){
\put(25.5,67){\scriptsize $Q^\vee$}
\put(25,2){\scriptsize $Q^\vee$}
\put(2,67){\scriptsize $Q$}
\put(2,2){\scriptsize $Q$}
\put(11,2){\scriptsize $L_h$}
\put(11,67){\scriptsize $L_h$}
\put(35.5,50.5){\scriptsize $f_h^{-\vee}$}
\put(35.5,19){\scriptsize $f_h^\vee$}
}
  }
\ep
}~=~\id_{Q\oti L_h\oti Q^\vee} \,. \quad
  \end{eqnarray}
Moreover, $F_h$ clearly intertwines the left actions of $A$ on $Q\oti L_h\oti Q^\vee$ 
and on $_{\id}A_{\alpha_h}$. That it intertwines the right actions as well is verified
as follows:
  \be
\raisebox{-40pt}{\rule{0pt}{87pt}} 
\raisebox{-30pt}{
\bp(95,77)
\bild{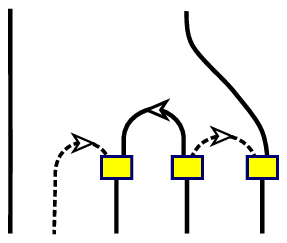}
\put(-40,-10){
\put(10,81){\scriptsize $Q^\vee$}
\put(31,2){\scriptsize $Q^\vee$}
\put(-12,2){\scriptsize $Q^\vee$}
\put(-42,81){\scriptsize $Q$}
\put(-43,2){\scriptsize $Q$}
\put(9,2){\scriptsize $Q$}
\put(-29,2){\scriptsize $L_h$}
\put(-22,27){\scriptsize $f_h^\vee$}
\put(40.5,29){\scriptsize $f_h^{\vee}$}
\put(14,20){\scriptsize $f_h^{-1}$}
}
\ep }
=~~~
\raisebox{-30pt}{
\bp(90,77)
\bild{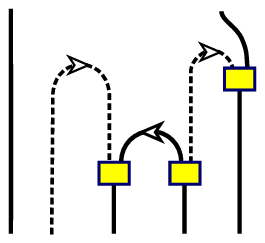}
\put(-40,-10){
\put(24,81){\scriptsize $Q^\vee$}
\put(31,2){\scriptsize $Q^\vee$}
\put(-7,2){\scriptsize $Q^\vee$}
\put(-38,81){\scriptsize $Q$}
\put(-38,2){\scriptsize $Q$}
\put(13,2){\scriptsize $Q$}
\put(-25,2){\scriptsize $L_h$}
\put(-17,27){\scriptsize $f_h^\vee$}
\put(40,53){\scriptsize $f_h^{\vee}$}
\put(19,18.5){\scriptsize $f_h^{-1}$}
}
\ep }
=~~~ \raisebox{-30pt}{
\bp(70,77)
\bild{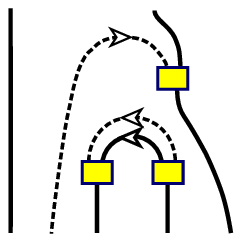}
\put(-40,-10){
\put(14,81){\scriptsize $Q^\vee$}
\put(34,2){\scriptsize $Q^\vee$}
\put(-4,2){\scriptsize $Q^\vee$}
\put(-29,81){\scriptsize $Q$}
\put(-29,2){\scriptsize $Q$}
\put(16,2){\scriptsize $Q$}
\put(-18,2){\scriptsize $L_h$}
\put(-11,17){\scriptsize $f_h^\vee$}
\put(27.3,53.5){\scriptsize $f_h^{\vee}$}
\put(21,18){\scriptsize $f_h^{-1}$}
}
\ep }
~~~=~~~
\raisebox{-30pt}{
\bp(92,77)
\bild{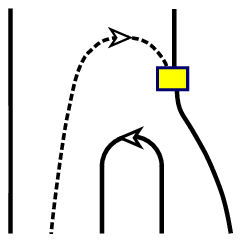}
\put(-40,-10){
\put(18,81){\scriptsize $Q^\vee$}
\put(34,2){\scriptsize $Q^\vee$}
\put(-3,2){\scriptsize $Q^\vee$}
\put(-30,81){\scriptsize $Q$}
\put(-30,2){\scriptsize $Q$}
\put(15,2){\scriptsize $Q$}
\put(-18,2){\scriptsize $L_h$}
\put(27.3,54){\scriptsize $f_h^{\vee}$}
}
\ep }
  \ee
Here similar steps are performed as in the proof that $\alpha_h$ respects the 
product of $A$. We conclude that we have $[Q\oti L_h\oti Q^\vee]\iN \im(\PsiA)$ 
for all $h\iN H$, and thus $\Pi_{Q,Q^\vee}(H)$ is a subgroup of $\im (\PsiA)$.
On the other hand, for $g\,{\not\in}\, H$, $Q\oti L_g\oti Q^\vee\cong 
\bigoplus_{h, h^\prime\in H}L_{hgh^\prime}$ is not isomorphic to $Q\oti Q^\vee$, 
not even as an object, so that $\Pi_{Q,Q^\vee}(g)\,{\not\in}\,\im(\PsiA)$. 
Together it follows that $\im(\Pi_{Q,Q^\vee} 
\cir\PsiA)\,{=}\, H$.
\end{proof}

\begin{rem}
Similarly to the calculation that the morphisms $\alpha_h$ in \erf{alphah} are 
morphisms of algebras, one shows that they also respect the coproduct and the 
counit of $A(H)$. So in fact we have found automorphisms of Frobenius algebras.
\end{rem}

We denote the inclusion morphisms $L_h\,\To Q \eq \bigoplus_{g\in H}L_g$ by 
$e_h$ and the projections $Q \To L_h$ by $r_h$, such that $r_g\cir e_h \eq 0$ 
for $g \,{\neq}\, h$ and $r_g\cir e_g \eq \id_{L_g}$. Then $e_g\cir r_g \eq P_g$ 
is a nonzero idempotent in $\End (Q)$, and we have $\sum_{h\in H}P_h \eq \id_Q$.

\begin{lemma}\label{alphagestalt}
Let $H\,{\le}\,\pic{D}$ be a finite subgroup such that $\dim(Q)\,{\ne}\,0$ for 
$Q \,{\equiv}\, Q(H)$.
Given $g \iN H$ and an automorphism $\alpha$ of $A \eq Q\oti Q^\vee$ such that 
$\,\PsiA(\alpha) \eq [Q\oti L_g\oti Q^\vee]$, there exists a unique isomorphism 
$f_g \iN \Hom(Q\Oti L_g,Q)$ such that $r_g \cir f_g \cir (e_\Eins \oti \id_{L_g})
\eq \id_{L_g}$ and $\alpha \eq \alpha_g$ with $\alpha_g$ as in \erf{alphah}.
\end{lemma}

\begin{proof}
We start by proving existence.
Let $\varphi_g{:}~ Q\oti L_g\oti Q^\vee\longcongTo {}_{\id}A_\alpha$ be an 
isomorphism of bimodules. As a first step we show that $\varphi_g \eq \id_Q\Oti\, h$ 
for some morphism $h{:}\ L_g\oti Q^\vee\congTo Q^\vee$:
  \be
\raisebox{-35pt}{
\bp(48,66)
   \put(0,10){
\bild{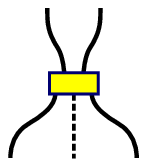}
\put(-40,-10){
\put(24,60){\scriptsize $Q^\vee$}
\put(35,2){\scriptsize $Q^\vee$}
\put(9,60){\scriptsize $Q$}
\put(-1,2){\scriptsize $Q$}
\put(17,2){\scriptsize $L_g$}
\put(2,32){\scriptsize $\varphi_g$}
   }
}
\ep
}~~=~\frac{1}{\dim (Q)} \quad
\raisebox{-35pt}{
\bp(35,53)
   \put(0,10){
\bild{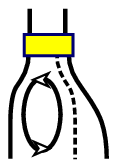}
\put(-40,-10){
\put(23,60){\scriptsize $Q^\vee$}
\put(36,2){\scriptsize $Q^\vee$}
\put(13,60){\scriptsize $Q$}
\put(7,2){\scriptsize $Q$}
\put(24,2){\scriptsize $L_g$}
\put(31,42){\scriptsize $\varphi_g$}
   }
}
\ep
}~~=~\frac{1}{\dim (Q)}\quad
\raisebox{-35pt}{
\bp(42,50)
   \put(0,10){
\bild{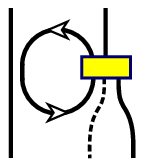}
\put(-40,-10){
\put(26,60){\scriptsize $Q^\vee$}
\put(36,2){\scriptsize $Q^\vee$}
\put(-1,60){\scriptsize $Q$}
\put(-1,2){\scriptsize $Q$}
\put(22,2){\scriptsize $L_g$}
\put(39.5,36){\scriptsize $\varphi_g$}
   }
}
\ep }
  \ee
Here in the first step we just inserted the dimension of $Q$, using that it is 
nonzero, and the second step is the statement that $\varphi_g$ intertwines the 
left action of $A$ on $Q\oti L_g\oti Q^\vee$ and on ${}_{\id}A_\alpha$.
Note that upon setting
  \be
f_g~:=~\frac{1}{\dim(Q)}~~
\raisebox{-30pt}{
\bp(40,63)
\put(0,10){
 \bild{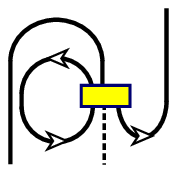}
 \put(-40,-10){
 \put(35.4,60){\scriptsize $Q$}
 \put(-11,2){\scriptsize $Q$}
 \put(16,2){\scriptsize $L_g$}
 \put(24,38){\scriptsize $\varphi_g$}
 }
}
\ep }
  \ee
this amounts to $\varphi_g \eq \id_Q {\Oti}\, ((\tilde d_{L_g}{\Oti}\,\id_{Q^\vee} )\cir
(\id_{L_g}\Oti\, f_g^\vee ))$, as in proposition \ref{mengenschnitt}.
\\[3pt]
Similarly the condition that $\varphi_g$ intertwines the right action of $A$ on 
$Q\oti L_g\oti Q^\vee$ and ${}_{\id}A_\alpha$ means that
  \be
\raisebox{-45pt}{
\bp(90,90)
  \put(0,14){
\bild{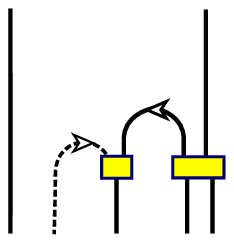}
\put(-40,-10){
\put(30,80){\scriptsize $Q^\vee$}
\put(34,2){\scriptsize $Q^\vee$}
\put(2,2){\scriptsize $Q^\vee$}
\put(-27,2){\scriptsize $Q$}
\put(-26.5,80){\scriptsize $Q$}
\put(24,2){\scriptsize $Q$}
\put(-14,2){\scriptsize $L_g$}
\put(-7.7,26.5){\scriptsize $f_g^\vee$}
\put(41,28){\scriptsize $\alpha$}
}
  }
\ep
}=\quad~~
\raisebox{-45pt}{
\bp(75,70)
  \put(0,14){
\bild{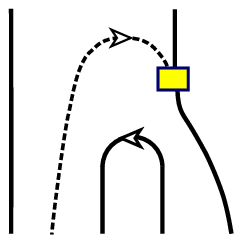}
\put(-40,-10){
\put(18,80){\scriptsize $Q^\vee$}
\put(34,2){\scriptsize $Q^\vee$}
\put(-4,2){\scriptsize $Q^\vee$}
\put(-29.5,2){\scriptsize $Q$}
\put(-29,80){\scriptsize $Q$}
\put(15,2){\scriptsize $Q$}
\put(-18,2){\scriptsize $L_g$}
\put(27,53.5){\scriptsize $f_g^\vee$}
}
  }
\ep }
  \ee
and this is equivalent to equality of the same pictures with the identity 
morphisms at the left sides removed.
Applying duality morphisms to both sides of the resulting equality we obtain
  \be
\raisebox{-37pt}{
\bp(46,76)
  \put(0,12){
\bild{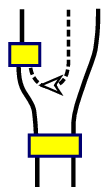}
\put(-40,-10){
\put(35,68){\scriptsize $Q^\vee$}
\put(29,2){\scriptsize $Q^\vee$}
\put(12,68){\scriptsize $Q$}
\put(17,2){\scriptsize $Q$}
\put(22,68){\scriptsize $L_g^\vee$}
\put(2.3,47.5){\scriptsize $f_g$}
\put(10.5,21.2){\scriptsize $\alpha$}
}
  }
\ep
}=\quad
\raisebox{-37pt}{
\bp(35,50)
  \put(0,12){
\bild{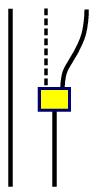}
\put(-40,-10){
\put(35,68){\scriptsize $Q^\vee$}
\put(25,2){\scriptsize $Q^\vee$}
\put(13,68){\scriptsize $Q$}
\put(13,2){\scriptsize $Q$}
\put(22,68){\scriptsize $L_g^\vee$}
\put(36,33.7){\scriptsize $f_g^\vee$}
}
  }
\ep }
  \ee
Next we apply the morphism $(\id_Q\oti\tilde d_{L_g}\oti\id_{Q^\vee})\cir 
(f_g^{-1}\oti \id_{L_g^\vee}\oti\id_{Q^\vee})$, leading to
  \be
\raisebox{-32pt}{
\bp(48,63)
  \put(0,12){
\bild{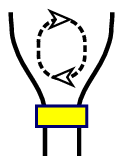}
\put(-40,-10){
\put(35,57){\scriptsize $Q^\vee$}
\put(25,2){\scriptsize $Q^\vee$}
\put(5,57){\scriptsize $Q$}
\put(14,2){\scriptsize $Q$}
\put(9.5,19.7){\scriptsize $\alpha$}
}
  }
\ep
}=\qquad
\raisebox{-32pt}{
\bp(35,50)
  \put(0,12){
\bild{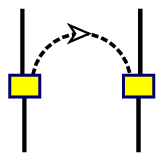}
\put(-40,-10){
\put(31,58){\scriptsize $Q^\vee$}
\put(31,2){\scriptsize $Q^\vee$}
\put(-3,58){\scriptsize $Q$}
\put(-2,2){\scriptsize $Q$}
\put(-18,29.5){\scriptsize $f_g^{-1}$}
\put(41.2,29){\scriptsize $f_g^\vee$}
}
  }
\ep }
  \ee
Using also that by lemma \ref{charakter}\,(ii) we have $\dimR(L_g) \eq 1$, this 
shows that $\alpha \eq \alpha_g$ with $\alpha_g$ as in \erf{alphah}. Since $L_g$ 
is absolutely simple (see lemma \ref{theabovelemma}\,(iv)),
we have $r_g \cir f_g \cir (e_\Eins \oti \id_{L_g}) \eq \xi\, \id_{L_g}$ for some 
$\xi \iN \Bbbk$. As $r_g$ is injective and $f_g$ is an isomorphism,
$r_g \cir f_g{:}~ Q \oti L_g \To L_g$ is nonzero. But this morphism can only be 
nonvanishing on the direct summand $\Eins$ of $Q$, and hence also 
$r_g \cir f_g \cir (e_\Eins \oti \id_{L_g}) \,{\neq}\, 0$, i.e.\ 
$\xi \iN \Bbbk^\times$. Finally note that $\alpha$ does not change if we replace 
$f_g$ by a non\-ze\-ro multiple of $f_g$; hence after suitable rescaling
$f_g$ obeys both $\alpha \eq \alpha_g$ and 
$r_g \cir f_g \cir (e_\Eins \oti \id_{L_g}) \eq \id_{L_g}$, thus proving existence.
\\[.3em]
To show uniqueness, suppose that there is another isomorphism 
$f'_g{:}~ Q \oti L_g \To Q$ such that $\alpha \eq \alpha'_g$ (where 
$\alpha'_g$ is given by \erf{alphah} with $f'_g$ instead of $f_g$) and 
$r_g \cir f_g' \cir (e_\Eins \oti \id_{L_g}) \eq \id_{L_g}$. Composing both sides 
of the equality $\alpha_g \eq \alpha'_g$ with $f_g \oti \id_{Q^\vee}$ from the 
right and taking a partial trace over $Q$ 
of the resulting morphism $Q \oti L_g \oti Q^\vee \To Q \oti Q^\vee$ gives
  \be
  \begin{array}{l}
  (d_Q \oti \id_{Q^\vee}) \circ 
  (\id_{Q^\vee} \Oti\, \alpha_g) \circ
  (\id_{Q^\vee} \Oti\, f_g \oti \id_{Q^\vee}) \circ
  (\tilde b_Q \oti \id_{L_g} \Oti\, \id_{Q^\vee})
  \\{}\\[-.8em]
  \hspace*{5.2em} =
  (d_Q \oti \id_{Q^\vee}) \circ 
  (\id_{Q^\vee} \Oti\, \alpha_g') \circ
  (\id_{Q^\vee} \Oti\, f_g \oti \id_{Q^\vee}) \circ
  (\tilde b_Q \oti \id_{L_g} \Oti\, \id_{Q^\vee})
  \end{array}
  \ee
Substituting the explicit form of $\alpha_g$ and $\alpha'_g$, and using that
$\dim(Q) \,{\neq}\, 0$ and that $L_g$ is absolutely simple, one finds that
$f_g \eq \lambda\, f_g'$ for some $\lambda \iN \Bbbk$. The normalisation conditions
$r_g \cir f_g \cir (e_\Eins \oti \id_{L_g})$ $=$ $\id_{L_g}$ and
$r_g \cir f_g' \cir (e_\Eins \oti \id_{L_g}) \eq \id_{L_g}$ then force
$\lambda \eq 1$, proving uniqueness.
\end{proof}

The construction of the automorphisms $\alpha_h$ presented in proposition 
\ref{mengenschnitt} and lemma \ref{alphagestalt} still depends on the choice of 
isomorphisms $f_h$ or $\varphi_h$. As each such automorphism gets mapped to 
$[Q\oti L_h\oti Q^\vee]$, due to exactness of the sequence \erf{sequenz} 
different choices of $f_h$ lead to automorphisms which differ only by inner 
automorphisms.  On the other hand, 
the mapping $h \,{\mapsto}\, \alpha_h$ need not be a homomorphism of groups 
for any choice of the isomorphisms $f_h$. 

In the following we will formulate necessary and sufficient conditions 
that $H\,{\le}\,\pic{D}$ must satisfy for the assignment $h \,{\mapsto}\, \alpha_h$
to yield a group homomorphism from $H$ to $\Aut (A)$. Recall that
$L_g \oti L_h \,{\cong}\, L_{gh}$. Thus by lemma \ref{theabovelemma}\,(iv) the spaces 
$\Hom (L_g\oti L_h, L_{gh})$ are one-dimensional (but there is no canonical choice 
of an isomorphism to the ground field $\Bbbk$). For each pair $g, h\iN \pic{D}$ we
select a basis isomorphism $_gb_h \iN \Hom (L_g\oti L_h, L_{gh})$.
We denote their inverses by $^gb^h \iN \Hom (L_{gh}, L_g\oti L_h)$, i.e.\
$^gb^h\cir {}_gb_h \eq \id_{L_g\otimes L_h}$ and $_gb_h\cir {}^gb^h \eq \id_{L_{gh}}$. 
For $g \eq 1$ we take ${}_1b_g$ and ${}_gb_1$ to be the identity, which is 
possible by the assumed strictness of $\D$.

For any triple $g_1,g_2,g_3 \iN \Pic(\D)$      
the collection $\{{}_gb_h\}$ of morphisms provides us with two bases of the 
one-dimensional space $\Hom (L_{g_1}\Oti L_{g_2}\Oti L_{g_3}, L_{g_1g_2g_3})$, 
namely with $_{g_1g_2}b_{g_3}\cir ({}_{g_1}b_{g_2}\oti \id_{L_{g_3}})$ as well as
$_{g_1}b_{g_2g_3}\cir (\id_{L_{g_1}} {\otimes}\, {}_{g_2}b_{g_3})$. 
These differ by a nonzero scalar $\psi (g_1, g_2, g_3) \iN \Bbbk$:
  \be\label{defkozykel}
  _{g_1g_2}b_{g_3}\circ (_{g_1}b_{g_2}\oti \id_{L_{g_3}})
  = \psi (g_1, g_2, g_3) \, {}_{g_1}b_{g_2g_3}\circ
  (\id_{L_{g_1}} \!{\otimes}\, {}_{g_2}b_{g_3}) \,.
  \ee
The pentagon axiom for the associativity constraints of $\D$ implies that $\psi$ 
is a three-cocycle on the group $\Pic(\D)$ with values in $\Bbbk^\times$
(see e.g.\ appendix E of \cite{mose3}, chapter 7.5 of \cite{FRke},
or \cite{yama7}). Any other choice 
of bases leads to a cohomologous three-co\-cycle.  Observe that by taking ${}_1b_h$ 
and ${}_hb_1$ to be the identity on $L_h$ the cocycle $\psi$ is normalised, i.e.\ 
satisfies $\psi (g_1, g_2, g_3) \eq 1$ as soon as one of the $g_i$ equals $1$.

\begin{lemma}\label{basisfrob}
For $g,h,k \iN \pic{D}$, the bases introduced above obey the relation
  \be
\raisebox{-24pt}{
\bp(51,53)
  \put(0,4){
\bild{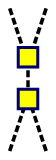}
\put(-40,-10){
\put(21,57){\scriptsize $L_k$}
\put(21,2){\scriptsize $L_h$}
\put(36,57){\scriptsize $L_{k^{-1}\!g}$}
\put(36,2){\scriptsize $L_{h^{-1}\!g}$}
\put(4,34){\scriptsize ${}^kb^{k^{-1}\!g}$}
\put(38,23){\scriptsize ${}_h^{}b_{h^{-1}\!g}$}
}
  }
\ep
}=~ \dsty\frac1{\psi(h,h^{-1}k,k^{-1}g)}\qquad~
\raisebox{-24pt}{
\bp(35,50)
  \put(0,4){
\bild{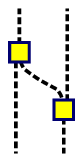}
\put(-40,-10){
\put(18,57){\scriptsize $L_k$}
\put(18,2){\scriptsize $L_h$}
\put(33,57){\scriptsize $L_{k^{-1}\!g}$}
\put(33,2){\scriptsize $L_{h^{-1}\!g}$}
\put(-6.5,39){\scriptsize ${}_h^{}b_{h^{-1}\!g}$}
\put(40.5,19){\scriptsize ${}^{h^{-1}\!k}b^{k^{-1}\!g}$}
}
  }
\ep }
  \ee
\end{lemma}

\begin{proof}
In pictures:
  \be
\raisebox{-35pt}{
\bp(47,82)
\bild{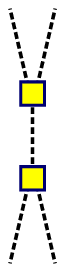}
\put(-40,-10){
\put(35,89){\scriptsize $L_{k^{-1}g}\!$}
\put(35,2){\scriptsize $L_{h^{-1}g}\!$}
\put(21,89){\scriptsize $L_k$}
\put(21,2){\scriptsize $L_h$}
\put(37,33){\scriptsize ${}_h^{}b_{h^{-1}\!g}$}
\put(2,56){\scriptsize ${}^kb^{k^{-1}\!g}$}
}
\ep
}=\qquad~~
\raisebox{-35pt}{
\bp(70,76)
\bild{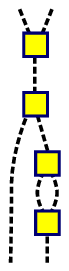}
\put(-40,-10){
\put(37,89){\scriptsize $L_{k^{-1}\!g}$}
\put(35,2){\scriptsize $L_{h^{-1}\!g}$}
\put(21,89){\scriptsize $L_k$}
\put(21,2){\scriptsize $L_h$}
\put(41,40){\scriptsize ${}_{h^{-1}\!k}b_{k^{-1}g\!}$}
\put(38,69){\scriptsize ${}^kb^{k^{-1}g}$}
\put(41,18){\scriptsize ${}^{h^{-1}\!k}b^{k^{-1}\!g}$}
\put(3,57){\scriptsize ${}_h^{}b_{h^{-1}\!g}$}
}
\ep }
  =~ \frac1{\psi}\qquad~
\raisebox{-45pt}{
\bp(68,76)
  \put(0,10){
\bild{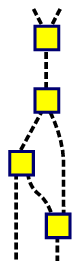}
\put(-40,-10){
\put(36,89){\scriptsize $L_{k^{-1}\!g}$}
\put(33,2){\scriptsize $L_{h^{-1}\!g}$}
\put(21,89){\scriptsize $L_k$}
\put(18,2){\scriptsize $L_h$}
\put(-4.5,37){\scriptsize ${}_h^{}b_{h^{-1}\!k}$}
\put(41,16){\scriptsize ${}^{h^{-1}k}b^{k^{-1}\!g}$}
\put(38,72){\scriptsize ${}^kb^{k^{-1}\!g}$}
\put(38,56){\scriptsize ${}_k^{}b_{k^{-1}\!g}$}
}
  }
\ep
 } =~ \frac1{\psi}\qquad~
\raisebox{-45pt}{
\bp(62,70)
  \put(0,10){
\bild{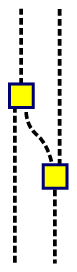}
\put(-40,-10){
\put(33,89){\scriptsize $L_{k^{-1}\!g}$}
\put(33,2){\scriptsize $L_{h^{-1}\!g}$}
\put(18,89){\scriptsize $L_k$}
\put(18,2){\scriptsize $L_h$}
\put(-5,59){\scriptsize ${}_h^{}b_{h^{-1}\!k}$}
\put(41,32){\scriptsize ${}^{h^{-1}\!k}b^{k^{-1}\!g}$}
}
  }
\ep }
   \ee
where in the second step we inserted the definition of $\psi$
and abbreviated $\psi \,{\equiv}\, \psi(h,h^{-1}k,k^{-1}g)$.
\end{proof}

\begin{Def}\label{trivialisation} 
Given the normalised cocycle $\psi$ on $\Pic(\D)$, a normalised two-cochain 
$\omega$ on $H$ with values in $\Bbbk^\times$ is called a {\em trivialisation\/} 
of $\psi$ on $H$ iff it satisfies $\rmd\omega \eq \psi\big|_H^{}$.
\end{Def}

\begin{prop}\label{trivialisierungsalgebra}
Given a finite subgroup $H$ of $\,\Pic(\D)$ and a function 
$\omega{:}~H\,{\times}\, H\To\Bbbk^\times$, define $Q \,{:=}\, \bigoplus_{h\in H}L_h$ and 
  \be \label{Q-morphs}
  \bearlll ~\\[-1.9em]
  m \!\!&\equiv\, m(H,\omega):=\sum_{g,h\in H}\omega(g,h)\, e_{gh}\cir{}_gb_h\cir(r_g\oti r_h)
  &\iN \Hom(Q\oti Q,Q) \,,
  \\[3pt]
  \eta \!\!&\equiv\,\eta(H,\omega):=e_1&\iN\Hom(\Eins, Q)\,,
  \\[3pt]
  \Delta \!\!&\equiv\,\Delta(H,\omega):={|H|}^{-1}\sum_{g,h\in H}{\omega(g,h)}^{-1}\,
  (e_g\oti e_h)\cir\, {}^gb^h\cir r_{gh}&\iN\Hom(Q,Q\oti Q)\,,
  \\[3pt]
  \epsilon \!\!&\equiv\,\epsilon(H,\omega):=|H|\ r_1&\iN\Hom(Q,\Eins)\,.
  \eear
  \ee
Then the following statements are equivalent:
\def\leftmargini{2.1em}
\begin{itemize}
\item[\rm (i)]   $\,\omega$ is a trivialisation of $\psi$ on $H$.
\item[\rm (ii)]  $\,(Q,m,\eta) $ is an associative unital algebra.
\item[\rm (iii)] $\,(Q,\Delta,\epsilon)$ is a coassociative counital coalgebra.
\end{itemize}
Moreover, if any of these equivalent conditions holds, then $Q(H,\omega) \,{\equiv}\, 
(Q,m,\Delta,\eta,\epsilon)$ is a special Frobenius algebra with 
$\,m\cir\Delta \eq \id_Q$ and $\,\epsilon\cir\eta \eq |H|\id_\Eins$.
\end{prop}

\begin{proof}
The equivalence of conditions (i)\,--\,(iii) follows by direct computation using 
only the definitions; we refrain from giving the details. The Frobenius property 
then follows with the help of lemma \ref{basisfrob}.
\end{proof}

\begin{lemma}
Let $\omega$ be a trivialisation of $\psi$ on $H \,{\le}\, \Pic(\D)$ and let
$Q\,{\equiv}\,\, Q(H,\omega)$ be the special Frobenius algebra defined 
in proposition \ref{trivialisierungsalgebra}.
Then $Q$ is symmetric iff $\,\dim(Q) \,{\neq}\, 0$.
\end{lemma}

\begin{proof}
Recall the morphisms $\Phi_1$ and $\Phi_2$ from \erf{Phi1Phi2}. Since
$Q \eq \bigoplus_{h\in H}L_h$, the condition $\Phi_1 \eq \Phi_2$ is equivalent 
to $\Phi_1\cir e_g \eq \Phi_2\cir e_g$ for all $g\iN H$. By the definition of 
the multiplication on $Q$, this amounts to
  \be
\omega(g,g^{-1})\qquad
\raisebox{-32pt}{
\bp(48,61)
  \put(0,12){
\bild{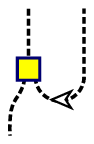}
\put(-40,-10){
\put(35,53){\scriptsize $L_{g^{-1}}^\vee$}
\put(19,52){\scriptsize $\Eins$}
\put(13,0){\scriptsize $L_g$}
\put(-3,31){\scriptsize ${}_g^{}b_{g^{-1}}$}
}
  }
\ep
}=~\omega(g^{-1},g)~~
\raisebox{-32pt}{
\bp(40,50)
  \put(0,12){
\bild{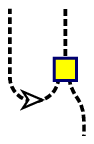}
\put(-40,-10){
\put(11,53){\scriptsize $L_{g^{-1}}^\vee$}
\put(32,52){\scriptsize $\Eins$}
\put(35,0){\scriptsize $L_g$}
\put(39.3,30){\scriptsize ${}_{g^{-1}}b_{g}$}
}
  }
\ep }
  \ee
which in turn, by applying duality morphisms and composing with the morphisms 
${}^gb^{g^{-1}}$, is equivalent to
  \be \label{Bild 53b}
\omega(g,g^{-1})\qquad
\raisebox{-30pt}{
\bp(32,66)
  \put(0,10){
\bild{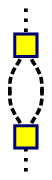}
\put(-40,-10){
\put(31.7,62){\scriptsize $\Eins$}
\put(31.7,2){\scriptsize $\Eins$}
\put(8.7,47){\scriptsize ${}_g^{}b_{g^{-1}}$}
\put(8.7,16){\scriptsize ${}^gb^{g^{-1}}$}
}
  }
\ep
}=~~\omega(g^{-1},g)~~
\raisebox{-30pt}{
\bp(40,52)
  \put(0,10){
\bild{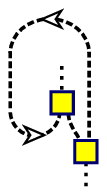}
\put(-40,-10){
\put(33.3,2){\scriptsize $\Eins$}
\put(40,35){\scriptsize ${}_{g^{-1}}b_{g}$}
\put(41.2,16){\scriptsize ${}^gb^{g^{-1}}$}
}
  }
\ep }
  \ee
The right hand side of \erf{Bild 53b} is evaluated to
  \begin{eqnarray}&&
\omega(g^{-1},g)\
\raisebox{-20pt}{
\bp(58,52)
\bild{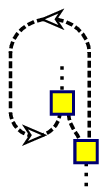}
\put(-40,-10){
\put(33,2){\scriptsize $\Eins$}
\put(39,35){\scriptsize ${}_{g^{-1}\!}b_g^{}$}
\put(41,17){\scriptsize ${}^{g_{\phantom:}\!\!}b^{g^{-1}}$}
}
\ep
}=~\omega(g^{-1},g)\,\psi(g^{-1},g,g^{-1})~~
\raisebox{-20pt}{
\bp(60,50)
\bild{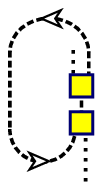}
\put(-40,-10){
\put(34.5,2){\scriptsize $\Eins$}
\put(41.5,26){\scriptsize ${}_{g^{-1}\!}b_{\Eins}^{}$}
\put(41.5,37){\scriptsize ${}^{\Eins\!}b^{g^{-1}}$}
}
\ep }
\nonumber\\&&\hspace*{3.3em}
=~\omega(g^{-1},g)\psi(g^{-1},g,g^{-1})~
\raisebox{-20pt}{
\bp(30,52)
\bild{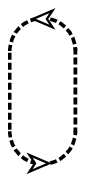}
\put(1,37){\scriptsize $L_g$}
\ep }
=~\omega(g^{-1},g)\,\psi(g^{-1},g,g^{-1})\,\dimL(L_g)\,\id_{\Eins} .\qquad
  \end{eqnarray}
The first step is an application of lemma \ref{basisfrob}, the second is due 
to our convention that the morphisms ${}^1b^g$ are chosen to be identity 
morphisms. So the condition that $Q$ is a symmetric Frobenius algebra is 
equivalent to the condition $\omega(g,g^{-1}) \eq \omega(g^{-1},g)\,
\psi(g^{-1},g,g^{-1})\,\dimL(L_g)$ for all $g\iN H$. As $\omega$ is a
trivialisation of $\psi$ and $\rmd\omega(g^{-1},g,g^{-1}) \eq \omega(g,g^{-1}) 
/ \omega(g^{-1},g)$, this is equivalent to $\dimL(L_g) \eq 1$ for all $g\iN H$.
By lemma \ref{charakter} the latter condition holds iff $\dim(Q) \,{\neq}\, 0$.
\end{proof}

\begin{Def} \label{admissible}
An \textit{admissible\/} subgroup of $\Pic(\D)$ is a finite subgroup 
$H\,{\le}\,\Pic(\D)$ such that $\dimLR(L_h) \eq 1$ for all $h\iN H$ and 
such that there exists a trivialisation $\omega$ of $\psi$ on $H$.
\end{Def}

\begin{rem}\label{trivalgebrarem}
We see that $Q(H,\omega)$ is a symmetric special Frobenius algebra if and only 
if $H$ is an admissible subgroup of $\Pic(\D)$ and $\omega$ is a trivialisation 
of $\psi$ on $H$. One can show that every structure of a special Frobenius algebra on 
the object $Q \eq \bigoplus_{h\in H}L_h$ is of the type $Q(H,\omega)$ described in 
proposition \ref{trivialisierungsalgebra} for a suitable trivialisation $\omega$ 
of $\psi$ (see \cite{fuRs9}, proposition 3.14). So giving product and coproduct 
morphisms on $Q$ is equivalent to giving a trivialisation $\omega$ of $\psi$. 
\\[.3em]
Also observe that multiplying a trivialisation $\omega$ of $\psi$ with a
two-cocycle $\gamma$ on $H$ gives another trivialisation $\omega'$ of $\psi$. One 
can show that the Frobenius algebras $Q(H,\omega)$ and $Q(H,\omega^\prime)$ are 
isomorphic as Frobenius algebras if and only if $\omega$ and $\omega^\prime$ 
differ by multiplication with an exact two-cocycle. Accordingly, in the 
sequel we call two trivialisations $\omega$ and $\omega^\prime$ for $\psi$ 
equivalent iff $\omega/\omega^\prime \eq \rmd\eta$ for some one-cochain $\eta$.
Thus if $\psi$ is trivialisable on $H$, then the equivalence classes of
trivialisations form a torsor over $H^2(H,\Bbbk^\times)$.
\end{rem}

\begin{thm}\label{gruppenschnitt}
Let $H$ be an admissible subgroup of $\pic D$, put $Q \eq Q(H)$ as in \erf{Q(H)},
and let $A\eq A(H) \eq Q \oti Q^\vee $ be the algebra defined in \erf{A(H)}. Then 
there is a bijection between
\\[3pt]
$\bullet$ \,trivialisations $\omega$ of $\psi$ on $H$ ~and
\\[3pt]
$\bullet$ \,group homomorphisms $\,\alpha{:}~ H\To \Aut (A)$ with 
$\,\Pi_{Q^\vee,Q}\cir\PsiA\cir\alpha \eq \id_H$.
\end{thm}

\begin{proof}
Denote by $\mathcal{T}$ the set of all trivialisations $\omega$ of $\psi$
on $H$, and by $\mathcal{H}$ the set of all group homomorphisms
$\alpha{:}~ H\To \Aut (A)$ satisfying $\Pi_{Q^\vee,Q}\cir\PsiA\cir\alpha \eq \id_H$.
The proof that $\mathcal{T} \,{\cong}\, \mathcal{H}$ as sets is organised in 
three steps: defining maps $F{:}~ \mathcal{T} \To \mathcal{H}$ and
$G{:}~ \mathcal{H} \To \mathcal{T}$, and showing that they are each
other's inverse.
\\[.3em]
(i)\, Let $\omega \iN \mathcal{T}$. For each $h \iN H$ define
  \be\label{T-H-proof-fh-def}
f_h ~:=~ \sum_{g\in H}\,\omega(g,h)~~
\raisebox{-37pt}{
\bp(35,75)
  \put(0,12){
\bild{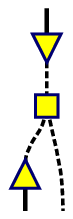}
\put(-40,-10){
\put(31.3,74){\scriptsize $Q$}
\put(36,2){\scriptsize $L_h$}
\put(24,2){\scriptsize $Q$}
\put(39,38){\scriptsize ${}_gb_h$}
\put(40,57){\scriptsize $e_{gh}$}
\put(14.8,20){\scriptsize $r_{g}$}
}}
\ep
}
  \ee
Since $\omega$ takes values in $\Bbbk^\times$, these are in fact isomorphisms,
with inverse given by
  \be\label{T-H-proof-fh-1-def}
  f_h^{-1} = \sum_{g\in H}\omega(g,h)^{-1}\,
  (e_g \oti \id_{L_h}) \cir {}^gb^h \cir r_{\!gh} \,.
  \ee
Define the function $F(\omega)$ from $H$ to $\Aut(A)$ by $F(\omega){:}\
h \mapsto \alpha_h$ for $\alpha_h$ given by \erf{alphah} with $f_h$ as in 
\erf{T-H-proof-fh-def}. We proceed to show that $F(\omega) \iN \mathcal{H}$.
\\[.3em]
Abbreviate $\alpha \,{\equiv}\, F(\omega)$.
That $\Pi_{Q^\vee,Q}\cir\PsiA\cir\alpha \eq \id_H$ follows from the proof
of proposition \ref{mengenschnitt}. To see that
$\alpha(g)\cir \alpha(h) \eq \alpha(gh)$, first rewrite $\alpha(g)\cir \alpha(h)$
using \erf{T-H-proof-fh-def} and \erf{T-H-proof-fh-1-def}:
  \be
\begin{array}{r} \displaystyle
\alpha(g)\cir\alpha(h)~=
\sum_{k,l,m,n}\frac{\omega(k,g)\,\omega(l,h)}{\omega(m,g)\,\omega(n,h)}\qquad
\raisebox{-45pt}{
\bp(68,104)
\bild{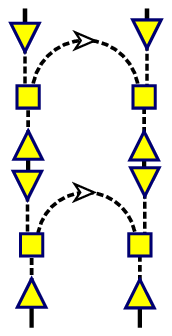}
\put(-39,-10){
\put(27,2){\scriptsize $Q^\vee$}
\put(30,108){\scriptsize $Q^\vee$}
\put(-5,108){\scriptsize $Q$}
\put(-3,2){\scriptsize $Q$}
\put(-20.5,74){\scriptsize ${}^mb^g$}
\put(-18.5,31.7){\scriptsize ${}^nb^h$}
\put(38,76){\scriptsize ${}_k\smash{b_g}^{\!\vee}$}
\put(37,33){\scriptsize ${}_l\smash{b_h}^{\!\vee}$}
\put(38.6,94){\scriptsize $\smash{r_k}^{\!\vee}$}
\put(37,48){\scriptsize $\smash{r_l}^{\!\vee}$}
\put(-21,62){\scriptsize $r_{mg}$}
\put(-19,18){\scriptsize $r_{nh}$}
\put(-17,93){\scriptsize $e_{m}$}
\put(-14,51){\scriptsize $e_{n}$}
\put(39,62){\scriptsize $\smash{e_{kg}}^{\!\vee}$}
\put(38.3,18){\scriptsize $\smash{e_{lh}}^{\!\vee}$}
\put(7,85){\scriptsize $L_{g}$}
\put(7,41){\scriptsize $L_{h}$}
}
\ep
}=~\sum_{k,m}\frac{\omega(k,g)\,\omega(kg,h)}{\omega(m,g)\,\omega(mg,h)}\qquad
\raisebox{-45pt}{
\bp(54,90)
\bild{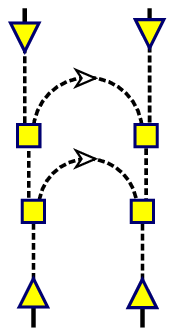}
\put(-40,-10){
\put(27,2){\scriptsize $Q^\vee$}
\put(30,108){\scriptsize $Q^\vee$}
\put(-5,108){\scriptsize $Q$}
\put(-3,2){\scriptsize $Q$}
\put(-20,62){\scriptsize ${}^mb^g$}
\put(-23,41){\scriptsize ${}^{mg}b^h$}
\put(39,63){\scriptsize ${}_k\smash{b_g}^{\!\vee}$}
\put(38,41){\scriptsize ${}_{kg}\smash{b_h}^{\!\vee}$}
\put(40,93){\scriptsize $\smash{r_k}^{\!\vee}$}
\put(-22,20){\scriptsize $r_{mgh}$}
\put(-17,94){\scriptsize $e_{m}$}
\put(39,18){\scriptsize $\smash{e_{kgh}}^{\!\vee}$}
\put(7,73){\scriptsize $L_{g}$}
\put(7,49.2){\scriptsize $L_{h}$}
}
\ep }
\\
\displaystyle
=~\sum_{k,m}~ \xi_{km} \qquad
\raisebox{-45pt}{
\bp(80,127)
\bild{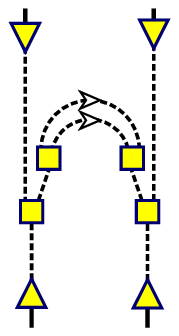}
\put(-40,-10){
\put(27,2){\scriptsize $Q^\vee$}
\put(30,108){\scriptsize $Q^\vee$}
\put(-5,108){\scriptsize $Q$}
\put(-4,2){\scriptsize $Q$}
\put(-17,60){\scriptsize ${}^gb^h$}
\put(-25,41){\scriptsize ${}^{m}b^{gh}$}
\put(37,60){\scriptsize ${}_g\smash{b_h}^{\!\vee}$}
\put(38,41){\scriptsize ${}_{k}\smash{b_{gh}}^{\!\vee}$}
\put(40,94){\scriptsize $\smash{r_k}^{\!\vee}$}
\put(-25,20){\scriptsize $r_{mgh}$}
\put(-18,93){\scriptsize $e_{m}$}
\put(39,18){\scriptsize $\smash{e_{kgh}}^{\!\vee}$}
}
\ep }
=~\sum_{k,m}~ \xi_{km} \qquad
\raisebox{-54pt}{
\bp(92,133)
  \put(0,9){
\bild{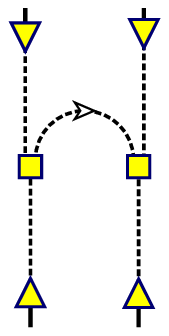}
\put(-38,-10){
\put(27,2){\scriptsize $Q^\vee$}
\put(30,108){\scriptsize $Q^\vee$}
\put(-5,108){\scriptsize $Q$}
\put(-3,2){\scriptsize $Q$}
\put(-24.3,55){\scriptsize ${}^{m}b^{gh}$}
\put(37,55){\scriptsize ${}_{k}\smash{b_{gh}}^{\!\vee}$}
\put(38,94){\scriptsize $\smash{r_k}^{\!\vee}$}
\put(-24,20){\scriptsize $r_{mgh}$}
\put(-18,94){\scriptsize $e_{m}$}
\put(37.3,18){\scriptsize $\smash{e_{kgh}}^{\!\vee}$}
\put(6,64){\scriptsize $L_{gh}$}
}
  }
\ep }
\label{alpha-alpha}
\end{array}
  \ee  
with
  \be
  \xi_{km} = \frac{\omega(k,g)\,\omega(kg,h)\,\psi(k,g,h)}
  {\omega(m,g)\,\omega(mg,h)\,\psi(m,g,h)} \,.
  \ee
Here the second step uses that there are no nonzero morphisms $L_n\To L_{mg}$ 
unless $n \eq mg$; by the same argument we conclude that $l \eq kg$. In the 
third step one applies relation \erf{defkozykel}. Now the condition 
$\alpha(g)\cir\alpha(h) \eq \alpha(gh)$ is equivalent to
  \be
  \frac{\omega(k,g)\,\omega(kg,h)\,\psi(k,g,h)} {\omega(m,g)\,\omega
  (mg,h)\,\psi(m,g,h)}
  = \frac{\omega(k,gh)}{\omega(m,gh)} ~\quad\mbox{for all }~ m,k\iN 
  H \,,
  \ee
which in turn can be rewritten as
  \be\label{domega-psi}
  \frac{\rmd\omega(m,g,h)}{\rmd\omega(k,g,h)} = \frac{\psi(m,g,h)}
  {\psi(k,g,h)}
  ~\quad\mbox{for all }~ m,k\iN H \,.
  \ee
The last condition is satisfied because by assumption $\rmd\omega \eq \psi|_H$.
So indeed we have $F(\omega) \iN \mathcal{H}$.
\\[.3em]
(ii)~\,Given $\alpha \iN \mathcal{H}$, for each $h \iN H$ the automorphism
$\alpha(h)$ satisfies the conditions of lemma \ref{alphagestalt}. As a
consequence we obtain a unique isomorphism $f_h{:}~ Q \oti L_h \To Q$ such 
that $\alpha(h)\eq \alpha_h$ and
$r_h \cir f_h \cir (e_\Eins \oti \id_{L_h}) \eq \id_{L_h}$. Define a
function $\omega{:}\ H \,{\times}\, H \To \Bbbk$ via
  \be
\omega(g,h)~
\raisebox{-37pt}{
\bp(42,81)
  \put(0,12){
\bild{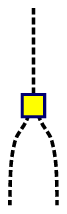}
\put(-40,-10){
\put(27,73){\scriptsize $L_{gh}$}
\put(37,2){\scriptsize $L_h$}
\put(22,2){\scriptsize $L_g$}
\put(38.5,38){\scriptsize ${}_gb_h$}
}}
\ep
} =~~~
\raisebox{-37pt}{
\bp(50,74)
  \put(0,12){
\bild{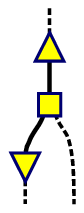}
\put(-40,-10){
\put(26,73){\scriptsize $L_{gh}$}
\put(35,2){\scriptsize $L_h$}
\put(20,2){\scriptsize $L_g$}
\put(11.7,20){\scriptsize $e_g$}
\put(37.5,37){\scriptsize $f_h$}
\put(14,55){\scriptsize $r_{gh}$}
}}
\ep } 
\label{omega-from-alpha}
  \ee
Then define the map $G$ from $\mathcal{H}$ to functions $H \,{\times}\, H
\To \Bbbk$ by $G(\alpha) \,{:=}\, \omega$, with $\omega$ obtained as in
\erf{omega-from-alpha}. We will show that $G(\alpha) \iN \mathcal{T}$.
\\[.3em]
Given $\alpha \iN \mathcal{H}$, abbreviate $\omega \,{\equiv}\, G(\alpha)$.
First note that $\omega$ takes values in $\Bbbk^\times$, as $f_h$ is an
isomorphism. Next, the normalisation condition
$r_h \cir f_h \cir (e_\Eins \oti \id_{L_h})\eq \id_{L_h}$
implies $\omega(1,h) \eq 1$ for all $h\iN H$. Since $\alpha$ is a group
homomorphism we have $\alpha(1) \eq \id_{Q \oti Q^\vee}$. By the uniqueness 
result of lemma \ref{alphagestalt} this implies that $f_1 \eq \id_Q$, and so 
$\omega(g,1)\eq 1$ for all $g \iN H$. Altogether it follows that $\omega$ is a 
normalised two-cochain with values in $\Bbbk^\times$. By following again
the steps \erf{alpha-alpha} to \erf{domega-psi} one shows that, since $\alpha$ 
is a group homomorphism, $\omega$ must satisfy \erf{domega-psi}. Setting 
$k\eq 1$ and using that $\omega$ and $\psi$ are normalised finally demonstrates 
that $\rmd\omega \eq \psi\big|_H^{}$. Thus indeed $G(\alpha) \iN \mathcal{T}$.
\\[.3em]
(iii)~\,That $F(G(\alpha))\eq\alpha$ is immediate by construction, and that
$G(F(\omega)) \eq \omega$ follows from the uniqueness result of lemma 
\ref{alphagestalt}.
\end{proof}

We have seen that if $H$ is an admissible subgroup of $\pic D$ and $\omega$ a 
trivialisation of $\psi$, then $Q(H,\omega)$ is a symmetric special 
Frobenius algebra. Using the product and coproduct morphisms of $Q(H,\omega)$, 
the automorphisms $\alpha_h$ induced by $\omega$ as described in theorem 
\ref{gruppenschnitt} can be written as
  \be\label{eq:alpha-via-Q}
\alpha_h~=|H|\quad
\raisebox{-28pt}{
\bp(80,75)
\put(0,10){
\bild{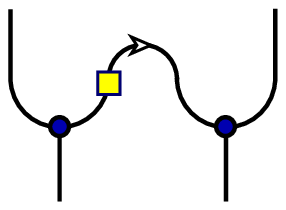}
\put(-40,-10){
\put(34,71){\scriptsize $Q^\vee$}
\put(19,1){\scriptsize $Q^\vee$}
\put(-42,71){\scriptsize $Q$}
\put(-28,1){\scriptsize $Q$}
\put(-24.5,43){\scriptsize $P_h$}
}}
\ep }
  \ee
Note that in this picture the circle on the left stands for the coproduct of $Q$,
while the circle on the right stands for the dual of the product.
    
Given a trivialisation $\omega$ of $\psi$, theorem \ref{gruppenschnitt} allows 
us to realise $H$ as a subgroup of $\Aut (A)$. In particular the fixed algebra 
under the action of $H$ is well defined; we denote it by $A^H$. We already 
know from proposition \ref{fixalgebraeigenschaften} that $A^H$ is a symmetric 
Frobenius algebra. It will turn out that it is isomorphic to $Q(H,\omega)$. 
In particular, by remark \ref{trivalgebrarem} any equivalent choice of a 
trivialisation $\omega'$ of $\psi$ will give an isomorphic fixed algebra.

\begin{thm}\label{fixalgebrabestimmung}
Let $H$ be an admissible subgroup of $\,\pic{D}$ with trivialisation $\omega$,  
put $A\eq A(H)$ as in \erf{A(H)}, and embed $H \To \Aut(A)$ 
as in theorem \ref{gruppenschnitt}. Then the fixed algebra 
$A^H$ is well defined and it is isomorphic to $Q(H,\omega)$. 
\end{thm}

\begin{proof}
Let $Q \,{\equiv}\,Q(H,\omega)$. By lemma \ref{charakter} we have $\dim(Q) \eq |H|$. 
Identify $H$ with its image in $\Aut(A)$ via the embedding $H \To \Aut(A)$ 
determined by $\omega$ as in theorem \ref{gruppenschnitt}. Now define morphisms 
$i{:}~ Q\To A$ and $s{:}~ A\To Q$ by
  \be
  i := (m \otimes \id_{Q^\vee}) \circ (\id_Q \otimes b_Q) \,, \qquad
  s := (\id_Q \otimes \tilde d_Q) \circ (\Delta \otimes \id_{Q^\vee}) \,.  
  \ee  
We have $s \cir i \eq \id_Q$, implying that $i$ is monic and $Q$ is a retract
of $A$. We claim that $i$ is the inclusion morphism
of the fixed algebra. Recall from \erf{fixalgebraprojektor} the definition 
$P \eq {|H|}^{-1}\sum_{h\in H}\alpha_h$ of the idempotent corresponding to the 
fixed algebra object. In the case under consideration, $P$ takes the form
  \be
P~=\quad
\raisebox{-30pt}{
\bp(80,78)
\put(0,12){
\bild{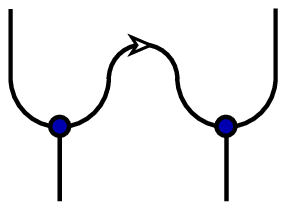}
\put(-40,-10){
\put(34,71){\scriptsize $Q^\vee$}
\put(19,1){\scriptsize $Q^\vee$}
\put(-42,71){\scriptsize $Q$}
\put(-28,1){\scriptsize $Q$}
}}
\ep }
  \ee
as follows from \erf{eq:alpha-via-Q} together with $\id_Q\eq\sum_{h\iN H}\!P_h$.
We calculate that $i\cir s \eq P$:
  \begin{eqnarray}
i\cir s~=~~
\raisebox{-60pt}{
\bp(70,137)
\bild{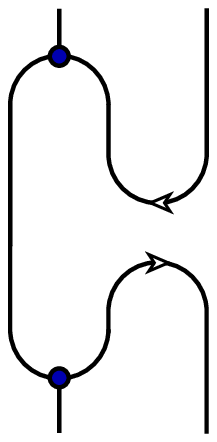}
\put(-38,-10){
\put(34,139){\scriptsize $Q^\vee$}
\put(33,1){\scriptsize $Q^\vee$}
\put(-9,139){\scriptsize $Q$}
\put(-10,1){\scriptsize $Q$}
}
\ep
}\overset{(2)}{=}~~ 
\raisebox{-60pt}{
\bp(90,130)
\bild{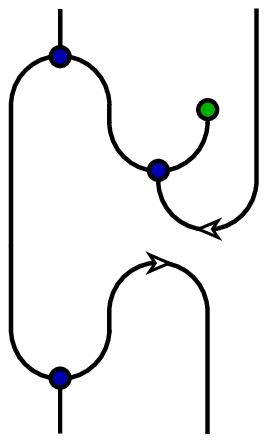}
\put(-39,-10){
\put(34,139){\scriptsize $Q^\vee$}
\put(20,1){\scriptsize $Q^\vee$}
\put(-22,139){\scriptsize $Q$}
\put(-22,1){\scriptsize $Q$}
}
\ep
}\overset{(3)}{=}~~
\raisebox{-60pt}{
\bp(95,130)
\bild{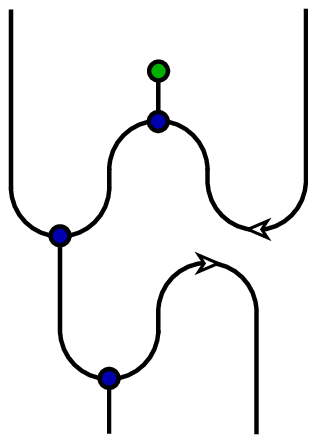}
\put(-39,-10){
\put(34,139){\scriptsize $Q^\vee$}
\put(20,1){\scriptsize $Q^\vee$}
\put(-51,139){\scriptsize $Q$}
\put(-23,1){\scriptsize $Q$}
}
\ep
}\overset{(4)}{=}~~
\raisebox{-60pt}{
\bp(78,130)
\bild{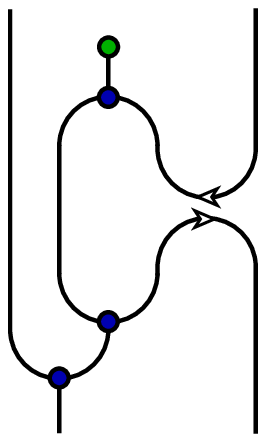}
\put(-39,-10){
\put(34,139){\scriptsize $Q^\vee$}
\put(34,1){\scriptsize $Q^\vee$}
\put(-36,139){\scriptsize $Q$}
\put(-23,1){\scriptsize $Q$}
}
\ep }
\nonumber\\
\overset{(5)}{=}~~
\raisebox{-64pt}{
\bp(91,164)
  \put(0,9){
\bild{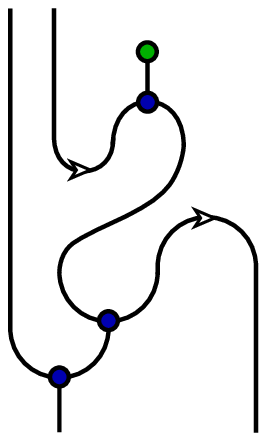}
\put(-40,-10){
\put(-22,139){\scriptsize $Q^\vee$}
\put(35,1){\scriptsize $Q^\vee$}
\put(-35,139){\scriptsize $Q$}
\put(-21,1){\scriptsize $Q$}
}
  }
\ep
}\overset{(6)}{=}~~
\raisebox{-64pt}{
\bp(91,160)
  \put(0,9){
\bild{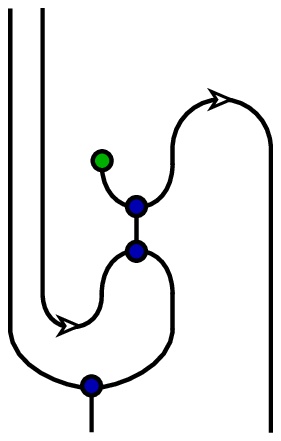}
\put(-40,-10){
\put(-29,139){\scriptsize $Q^\vee$}
\put(34,1){\scriptsize $Q^\vee$}
\put(-40,139){\scriptsize $Q$}
\put(-17,1){\scriptsize $Q$}
}
  }
\ep
}\overset{(7)}{=}~~
\raisebox{-64pt}{
\bp(101,160)
  \put(0,9){
\bild{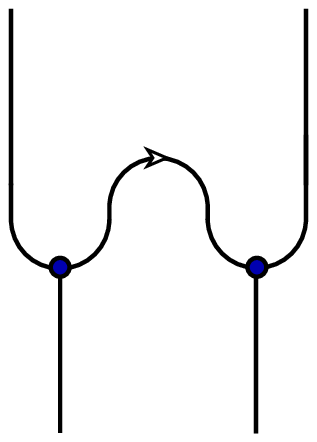}
\put(-40,-10){
\put(34,138){\scriptsize $Q^\vee$}
\put(19,1){\scriptsize $Q^\vee$}
\put(-50,138){\scriptsize $Q$}
\put(-36,1){\scriptsize $Q$}
}
  }
\ep
}~=P\,.\qquad
  \end{eqnarray}
Here in the second step a counit morphism is introduced and in the third step 
the Frobenius property is applied. The next step uses coassociativity, while
the fifth step follows because $Q(H,\omega)$ is symmetric.
Then one uses the Frobenius property and duality. So $Q$ satisfies the universal 
property of the image of $P$ and hence the universal property of the fixed 
algebra. We now calculate the product morphism that the fixed algebra inherits 
from $Q\oti Q^\vee$, starting from \erf{mP-from-er}:
  \begin{gather}
\raisebox{-65pt}{
\bp(141,142)
\bild{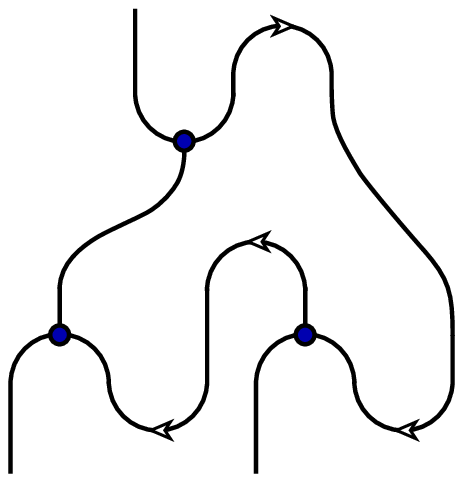}
\put(-40,-10){
\put(-21,2){\scriptsize $Q$}
\put(-56,150){\scriptsize $Q$}
\put(-93,2){\scriptsize $Q$}
}
\ep
}=~~
\raisebox{-65pt}{
\bp(93,140)
\bild{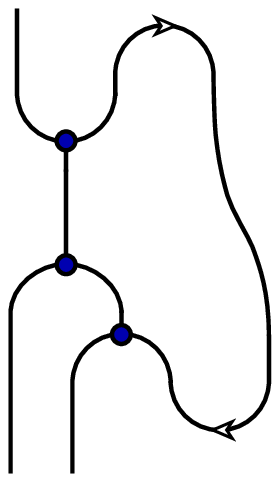}
\put(-40,-10){
\put(-39,2){\scriptsize $Q$}
\put(-38,150){\scriptsize $Q$}
\put(-23,2){\scriptsize $Q$}
}
\ep
}=~~
\raisebox{-65pt}{
\bp(89,140)
\bild{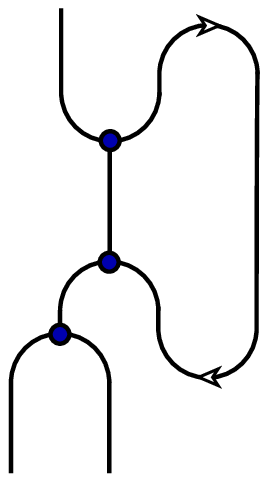}
\put(-40,-10){
\put(-8,2){\scriptsize $Q$}
\put(-21,150){\scriptsize $Q$}
\put(-37,2){\scriptsize $Q$}
}
\ep
}=~~
\raisebox{-65pt}{
\bp(81,140)
\bild{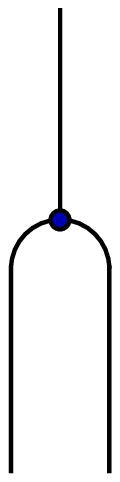}
\put(-40,-10){
\put(21,150){\scriptsize $Q$}
\put(34,2){\scriptsize $Q$}
\put(6,2){\scriptsize $Q$}
}
\ep
}
  \end{gather}
The first step uses duality, the second one holds by associativity of $m$. In 
the last step one uses that $s\cir i \eq \id_Q$. The inherited unit morphism is given by
  \be
\raisebox{-35pt}{
\bp(40,81)
\bild{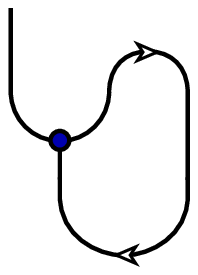}
\put(-40,-10){
\put(-15,90){\scriptsize $Q$}
}
\ep }
  \ee
which due to $\dimf \Hom (\Eins, Q) \eq 1$ is equal to $\zeta\eta$ for some 
$\zeta\iN\Bbbk$. Applying $\epsilon$ to both these morphisms and using that 
$\dim(Q) \eq |H|$, we see that $\zeta \eq 1$. Similarly one shows that the 
coproduct and counit morphisms that $Q$\ inherits as a fixed algebra equal 
those defined in proposition \ref{trivialisierungsalgebra}. So $Q(H,\omega)$ 
is isomorphic to the fixed algebra $A^H$ as a Frobenius algebra. 
\end{proof}

\begin{rem}
The algebra structure on the object $Q(H,\omega)$ is a kind of twisted 
group algebra of the group $H$ which is not twisted by a closed two-cochain, 
but rather by a trivialisation of the associator of $\D$. Algebras of this 
type have appeared in applications in conformal field theory \cite{fuRs9}.
\end{rem}


\section{Algebras in general Morita classes}\label{sec:gen-class}

In this section we solve the problem discussed in the previous section for algebras 
that are not Morita equivalent to the tensor unit. 
Throughout this section we will assume the following.
    
\begin{conv}\label{convention2}
$(\C, \oti, \Eins)$ has the properties listed in convention \ref{convention1}
and is in addition skeletally small and sovereign. $(A,m,\eta,\Delta,\epsilon)$ 
is a simple and absolutely simple symmetric normalised special 
Frobenius algebra in $\C$, and $H\,{\le}\, \Pic (\CAA)$ is a finite subgroup.
\end{conv}

Recall from remark \ref{frobdim} that the conditions above imply 
$\dim(A)\,{\neq}\, 0$. In the sequel we will find 
a symmetric special Frobenius algebra $A^\prime \eq A^\prime (H)$ and a Morita 
context $\MK {A}{A^\prime}{P}{P^\prime}$ in $\C$, such that $H$ is a subgroup of 
$\im(\Pi_{P,P^\prime}\cir\Psi_{\!A^\prime})$, where $\Pi_{P,P^\prime}$ is the 
isomorphism introduced in \erf{piciso}. This generalises the results of proposition 
\ref{mengenschnitt}.

We will apply some of the results of the previous section to the strictification 
$\D$ of the category $\CAA$. Note that $\D$ has the properties stated in 
convention \ref{convention1} and is in addition sovereign, as can be seen 
by straightforward calculations which are parallel to those of \cite{fuSc16} 
for the category $\CA$ of left $A$-modules. By applying the inverse
equivalence functor $\D\,{\stackrel{\simeq}{\longrightarrow}}\, \CAA$ this will 
then yield a symmetric special Frobenius algebra in $\CAA$ that has the desired 
properties. Note that the graphical representations of morphisms used below are 
meant to represent morphisms in $\C$. Pertinent facts about the structure of the 
category $\CAA$ are collected in appendix \ref{appA}; in the sequel we will freely 
use the terminology presented there. It is worth emphasising that for establishing 
various of the results below, it is essential that $A$ is not just an algebra in 
$\C$, but even a simple and absolutely simple symmetric special Frobenius algebra.

As a first step we study how concepts like algebras and modules over algebras 
can be transported from $\CAA$ to $\C$. If $X$ is an object of $\CAA$, 
i.e.\ an $A$-bimodule in $\C$, we denote the corresponding object of $\C$ by $\dot X$.

\begin{prop}\label{algebratransport}
~\\[-1.8em]
\def\leftmargini{2.1em} \begin{itemize}
\item[\rm (i)]
Let $(B, m_B, \eta_B)$ be an algebra in $\CAA$. 
Then $(\dot B, m_B\cir r_{B,B}, \eta_B\cir\eta)$ is an algebra in $\C$. 
\item[\rm (ii)] 
If $(C, \Delta_C, \epsilon_C)$ is a coalgebra in $\CAA$, then 
$(\dot C,e_{C,C}\cir\Delta_C, \epsilon\cir\epsilon_C)$ is a coalgebra in $\C$.
\item[\rm (iii)] 
A morphism $\gamma{:}~B\To\,B^\prime$ of algebras in $\CAA$ is also a morphism 
of algebras in $\C$.
\item[\rm (iv)] 
Let $(B,m_B,\eta_B)$ be an algebra in $\CAA$ and $(M, \rho)$ be a left $B$-module 
in $\CAA$. Then $(\dot M,\rho\cir r_{B,M})$ is a left $\dot B$-module in $\C$. 
Similarly right $B$-modules and $B$-bimodules in $\CAA$ can be transported to $\C$. 
Further, if $f{:}~(M,\rho_M)\To (N,\rho_N)$ is a morphism of left $B$-modules in 
$\CAA$, then $f$ is also a morphism of left $\dot B$-modules in $\C$, and an
analogous statement holds for morphisms of right- and bimodules.
\item[\rm (v)] 
If $(B, m_B, \Delta_B, \eta_B,\epsilon_B)$ is a Frobenius algebra in $\CAA$, then 
$(\dot B,m\cir r_{B,B},e_{B,B}\cir\Delta, \eta_B\cir\eta, \epsilon\cir\epsilon_B)$ 
is a Frobenius algebra in $\C$. If $B$ is special in $\CAA$, then $\dot B$ is 
special in $\C$. If $B$\ is symmetric in $\CAA$, then $\dot B$ is symmetric in $\C$.
\end{itemize}
\end{prop}

\begin{proof}
(i)~\,That $m_B$ is an associative product for $B$\ in $\CAA$ means that
  \be
  m_B\cir (\id_B\otimesA m_B)\cir \alpha_{B,B,B}=m_B\cir(m_B\otimesA\id_B) \,,
  \ee
where $\alpha_{B,B,B}$ is the associator defined in \erf{bimod-associator}.
After inserting the definitions of the tensor product of morphisms and 
the associator $\alpha_{B,B,B}$ this reads 
  \be
\raisebox{-63pt}{
\bp(65,133)
  \put(0,10){
\bild{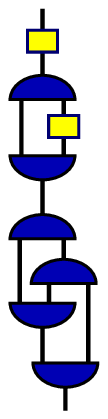}
\put(-40,-10){
\put(5,1){\scriptsize $(\bb)\otimesA B$}
\put(20,131){\scriptsize $B$}
\put(30.5,115.7){\scriptsize $m_B^{}$}
\put(36.5,90.7){\scriptsize $m_B^{}$}
\put(-6,103){\scriptsize $r_{B,B}^{}$}
\put(-23,63){\scriptsize $r_{B,\bb}^{}$}
\put(40,51){\scriptsize $r_{B,B}^{}$}
\put(-5,36){\scriptsize $e_{B,B}^{}$}
\put(-23,81){\scriptsize $e_{B,\bb}^{}$}
\put(-15,20){\scriptsize $e_{\bb,B}^{}$}
}
  }
\ep
}=\qquad~~
\raisebox{-63pt}{
\bp(35,100)
  \put(0,10){
\bild{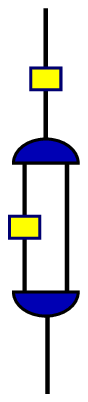}
\put(-40,-10){
\put(5,1){\scriptsize $(\bb)\otimesA B$}
\put(26.7,127){\scriptsize $B$}
\put(9,101){\scriptsize $m_B^{}$}
\put(3,58){\scriptsize $m_B^{}$}
\put(40,34){\scriptsize $e_{\bb,B}^{}$}
\put(40.3,81){\scriptsize $r_{B,B}^{}$}
}
  }
\ep }
  \ee
After composing both sides of this equality with the morphism 
$r_{B\otimesA B,B}\cir (r_{B,B}\oti\id_B)$ the resulting idempotents 
$P_{B,B\otimesA B}$, $P_{B\otimesA B,B}$ and $P_{B,B}$ can be dropped
from the left hand side, and $P_{B\otimesA B,B}$ from the right hand side.
We see that $m_B\cir r_{B,B}$ is indeed an associative product for 
$\dot B$ in $\C$. Next consider the morphism $\eta_B$; we have
  \begin{eqnarray}
\raisebox{-60pt}{
\bp(56,130)
\bild{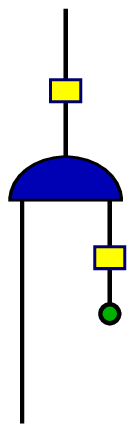}
\put(-40,-10){
\put(5.5,2){\scriptsize $B$}
\put(18.3,135.3){\scriptsize $B$}
\put(29.5,104){\scriptsize $m_B^{}$}
\put(41.5,56.5){\scriptsize $\eta_{B}^{}$}
\put(-13,81){\scriptsize $r_{\!B,B}^{}$}
}
\ep
}=\quad
\raisebox{-60pt}{
\bp(72,130)
\bild{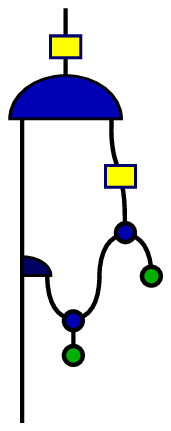}
\put(-40,-10){
\put(-5,2){\scriptsize $B$}
\put(8.5,135){\scriptsize $B$}
\put(18,117.5){\scriptsize $m_B^{}$}
\put(29.3,102){\scriptsize $r_{\!B,B}^{}$}
\put(35,80.5){\scriptsize $\eta_{B}^{}$}
}
\ep
}=\quad
\raisebox{-60pt}{
\bp(65,130)
\bild{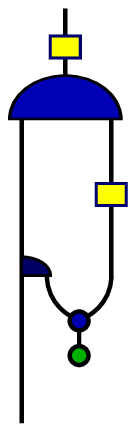}
\put(-40,-10){
\put(4.5,2){\scriptsize $B$}
\put(18.5,135){\scriptsize $B$}
\put(28,117){\scriptsize $m_B^{}$}
\put(39.3,102){\scriptsize $r_{\!B,B}^{}$}
\put(41,74.5){\scriptsize $\eta_{B}^{}$}
}
\ep
}=\quad
\raisebox{-60pt}{
\bp(122,130)
\bild{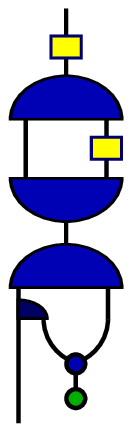}
\put(-40,-10){
\put(5,2){\scriptsize $B$}
\put(20,135){\scriptsize $B$}
\put(30,118){\scriptsize $m_B^{}$}
\put(41,103){\scriptsize $r_{\!B,B}^{}$}
\put(41.4,89.2){\scriptsize $\eta_{B}^{}$}
\put(40,74){\scriptsize $e_{B,A}^{}$}
\put(40,56){\scriptsize $r_{\!B,A}^{}$}
}
\ep }
\nonumber\\[1.6em]
=~ m_B\cir(\id_B\otimesA\,\eta_B)\cir\rho^A(B)^{-1}\,=\,\id_B \,,~ 
 \end{eqnarray}
with $\rho^A$ the unit constraint as given by \erf{Bild38} in the appendix. Here 
in the first step the idempotent $P_{B,B}$ is introduced and then moved downwards, 
and likewise in the third step. The fifth step is the unit property of $\eta_B$ 
in $\CAA$. So we see that $\eta_B\cir \eta$ is indeed a right unit for $\dot B$ 
in $\C$, similarly one shows that it is also a left unit.
\\[.5em]
(ii)\,
is proved analogously to the preceding statement.
\\[.5em]
(iii)\,
Let $m_B$ and $m_{B^\prime}$ denote the products of $B$ and $B^\prime$ in $\CAA$. Then
  \be
  \bearll
  \gamma\cir m_B\cir r_{B,B} \!\!& = m_{B^\prime}\cir (\gamma\otA\gamma)\cir r_{B,B}
  = m_{B^\prime}\cir r_{\!B^\prime,B^\prime}^{}\cir (\gamma\oti\gamma)\cir P_{B,B}
  \\{}\\[-.9em]&
  = m_{B^\prime}\cir r_{\!B^\prime,B^\prime}^{}\cir P_{\!B^\prime,B^\prime}^{}\cir(\gamma\oti\gamma)
  = m_{B^\prime}\cir r_{\!B^\prime,B^\prime}^{}\cir (\gamma\oti\gamma)\,,
  \eear
  \ee
where the third equality uses that $\gamma$ is a morphism in $\CAA$. Further we have
$\gamma\cir \eta_B\cir\eta \eq \eta_{B^\prime}\cir\eta$, as $\gamma$ respects the 
unit $\eta_B$ of $B$ in $\CAA$. So $\gamma$ is also a morphism of algebras in $\C$.
\\[.5em]
(iv)\,
The statement that $M$ is a left $B$-module in $\CAA$ reads
  \be
  \rho\cir(\id_B\otimesA\rho)\cir\alpha_{B,B,M}=\rho\cir (m_B\otimesA\id_M) \,,
  \ee
which is an equality in $\Hom_{A|A}((B\otimesA\,B)\otimesA M,M)$. Similarly to the 
proof in i), one shows that this indeed implies that $(M,\rho\cir r_{B,M})$ is a 
left $\dot B$-module in $\C$.
\\[3pt]
Now for any morphism $f{:}~(M,\rho_M)\To (N,\rho_N)$ we have
  \be
  \bearll
  f\cir \rho_M\cir r_{B,M} \!\!& =\rho_N\cir(\id_B\otimesA\, f)\cir r_{B,M}
  =\rho_N\cir r_{B,N}\cir (\id_B\oti f)\cir P_{B,M}
  \\{}\\[-.9em]&
  =\rho_N\cir r_{B,N}\cir P_{B,N}\cir (\id_B\oti f)
  =\rho_N\cir r_{B,N}\cir(\id_B\oti f) \,,
  \eear
  \ee
showing that $f$ is a morphism of left $\dot B$-modules in $\C$. Similarly one 
verifies the conditions for right- and bimodules. 
\\[.5em]
(v)\,
To see that the product and coproduct morphisms for $\dot B$ satisfy the 
Frobenius property in $\C$ consider the following calculation:
  \be
\raisebox{-83pt}{
\bp(53,164)
  \put(0,8){
\bild{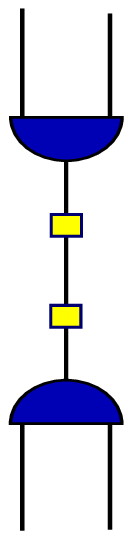}
\put(-40,-10){
\put(6,2){\scriptsize $B$}
\put(32,2){\scriptsize $B$}
\put(32,165){\scriptsize $B$}
\put(6,165){\scriptsize $B$}
\put(30,70){\scriptsize $m_B^{}$}
\put(30,96){\scriptsize $\Delta_{B}$}
\put(-12,122){\scriptsize $e_{B,B}^{}$}
\put(-12,46){\scriptsize $r_{\!B,B}^{}$}
}
  }
\ep
}=\qquad\quad
\raisebox{-83pt}{
\bp(85,150)
  \put(0,8){
\bild{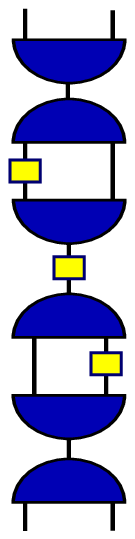}
\put(-40,-10){
\put(7,2){\scriptsize $B$}
\put(7,165){\scriptsize $B$}
\put(33,2){\scriptsize $B$}
\put(33,165){\scriptsize $B$}
\put(-9,113){\scriptsize $m_B^{}$}
\put(40,56){\scriptsize $\Delta_B$}
\put(41,21){\scriptsize $r_{\!B,B}^{}$}
\put(41,41){\scriptsize $e_{B,B}^{}$}
\put(-28,71){\scriptsize $r_{\!B,\bb}^{}$}
\put(41,99){\scriptsize $e_{\bb,B}^{}$}
\put(41,145){\scriptsize $e_{B,B}^{}$}
\put(41,126){\scriptsize $r_{\!B,B}^{}$}
\put(29.6,84){\scriptsize $\alpha_{B,B,B}^{-1}$}
}
  }
\ep
}=\qquad~~
\raisebox{-83pt}{
\bp(83,150)
  \put(0,8){
\bild{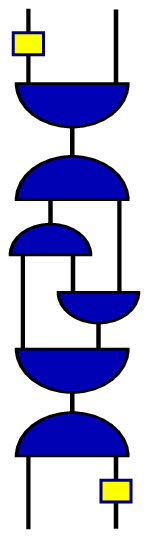}
\put(-40,-10){
\put(2,2){\scriptsize $B$}
\put(3,165){\scriptsize $B$}
\put(29,2){\scriptsize $B$}
\put(29,165){\scriptsize $B$}
\put(-12.7,149){\scriptsize $m_B^{}$}
\put(38,36){\scriptsize $r_{B,\bb}^{}$}
\put(37,131){\scriptsize $e_{\bb,B}^{}$}
\put(38,109){\scriptsize $r_{\!\bb,B}^{}$}
\put(-17,92){\scriptsize $r_{\!B,B}^{}$}
\put(36,53){\scriptsize $e_{B,\bb}^{}$}
\put(41,73){\scriptsize $e_{B,B}^{}$}
\put(39,18){\scriptsize $\Delta_{B}$}
}
  }
\ep
}=\qquad
\raisebox{-83pt}{
\bp(65,150)
  \put(0,8){
\bild{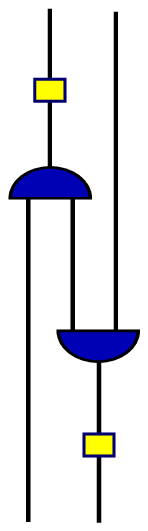}
\put(-40,-10){
\put(2,2){\scriptsize $B$}
\put(8.5,164){\scriptsize $B$}
\put(25,2){\scriptsize $B$}
\put(29,164){\scriptsize $B$}
\put(-8,134){\scriptsize $m_B^{}$}
\put(-18,108){\scriptsize $r_{\!B,B}^{}$}
\put(34,30.7){\scriptsize $\Delta_{B}$}
\put(39,58){\scriptsize $e_{B,B}^{}$}
}
  }
\ep }
  \ee
The first equality is the assertion that $B$ is a Frobenius algebra in $\CAA$, 
the second one implements the definition of $\alpha_{B,B,B}^{-1}$. In the last 
step the resulting idempotents are moved up or down, upon which they can be
dropped. A parallel argument establishes the second identity.
\\[3pt]
Similarly one checks that specialness and symmetry of $\dot B$ are transported 
to $\C$ as well.
\end{proof}

To apply the results of the previous section to the algebra $A$ we also need 
to deal with Morita equivalence in $\CAA$. We start with the following observation.

\begin{lemma}\label{tensortransport}
Let $B$ be a symmetric special Frobenius algebra in $\,\CAA$, 
and $C$ and $D$ be algebras in $\,\CAA$ and let $({}_CM_B,\rho_C,\varrho_B)$ 
be a $C$-$B$-bimodule and $({}_BN_D,\rho_B,\varrho_D)$ a $B$-$D$-bi\-module. 
Then the tensor product $M \,{\otimes_B}\, N$ in $\CAA$ is isomorphic, 
as a $\dot C$-$\dot D$-bimodule in $\C$, to the tensor product 
$\dot M \,{\otimes_{\dot B}}\, {\dot N}$ over the algebra $\dot B$ in $\C$.
\end{lemma}

\begin{proof}
Since $B$ is symmetric special Frobenius in $\CAA$, the idempotent $P_{M,N}^B$ 
corresponding to the tensor product of $M$ and $N$ over $B$ is well defined in 
$\CAA$. Explicitly it reads
  \be
  (\varrho_B\otimesA\rho_B)\cir\alpha_{M,B,B\otimesA N}^{-1}\cir (\id_M\otimesA\alpha_{B,B,N})
  \cir (\id_M\otimesA(\Delta_B\cir\eta_B)\otimesA\id_N))\cir(\id_M\otimesA\lambda^{A}(N)^{-1})\,.
  \ee
One calculates that this equals the morphism given by the following morphism in $\C$:
  \be
\raisebox{-83pt}{
\bp(55,149)
  \put(0,8){
\bild{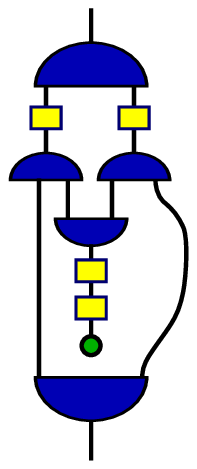}
\put(-45,-10){
\put(3,2){\scriptsize $M\otimesA N$}
\put(3,147){\scriptsize $M\otimesA N$}
\put(-14,108.5){\scriptsize $\varrho_B^{}$}
\put(34.7,108.5){\scriptsize $\rho_B^{}$}
\put(23,53){\scriptsize$\eta_B^{}$}
\put(22.6,63){\scriptsize$\Delta_{B}$}
\put(31,24){\scriptsize $e_{M,N}^{}$}
\put(24,72){\scriptsize$e_{B,B}$}
\put(41,93){\scriptsize $r_{\!B,N}^{}$}
\put(-28,93){\scriptsize $r_{\!M,B}^{}$}
\put(34,124){\scriptsize $r_{\!M,N}^{}$}
}
  }
\ep }
  \ee
Composing with the morphisms $e_{M,N}$ and $r_{M,N}$ for the tensor product over 
$A$, we see that $P_{\dot M,\dot N}^{\dot B} \eq e_{M,N}\cir P_{M,N}^{B}
\cir r_{M,N} \eq e_{M,N}\cir e_{M,N}^B\cir r_{M,N}^B\cir r_{M,N}$. This furnishes
a different decomposition of $P_{\dot M,\dot N}^{\dot B}$ into a monic and an epi, 
hence there is an isomorphism $f{:}~ M\otimes_BN\congTo \dot M\otimes_{\dot B}\dot N$ 
of the images of $P_{M,N}^B$ and $P_{\dot M,\dot N}^{\dot B}$, such that 
$f\cir r_{M,N}^B\cir r_{M,N} \eq r_{\dot M,\dot N}^{\dot B}$ and 
$e_{M,N}\cir e_{M,N}^B \eq e_{\dot M,\dot N}^{\dot B}\cir f$. Now the left action of 
$C$ on $M\otimes_B N$ is given by
  \be
r_{M,N}^B\cir (\rho_C\otimesA\id_N)\cir\alpha_{C,M,N}^{-1}\cir
  (\id_C\otimesA\, e_{M,N}^B)~=\qquad~
\raisebox{-85pt}{
\bp(87,182)
  \put(0,10){
\bild{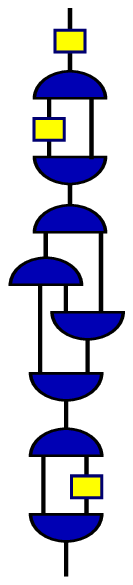}
\put(-40,-10){
\put(-3,1){\scriptsize $C\otimesA(M\otimesB N)$}
\put(12,181){\scriptsize $M\otimesB N$}
\put(1,138){\scriptsize $\rho_C^{}$}
\put(31,163){\scriptsize $r_{M,N}^B$}
\put(36,126){\scriptsize $e_{C\otimesA M,N}^{}$}
\put(36,110){\scriptsize $r_{\!C\otimesA M,N}^{}$}
\put(34,62){\scriptsize $e_{C,M\otimesA N}^{}$}
\put(35,48){\scriptsize $r_{C,M\otimesA N}^{}$}
\put(41,80){\scriptsize $e_{M,N}^{}$}
\put(-13,99){\scriptsize $r_{\!C,M}^{}$}
\put(37,149){\scriptsize $r_{M,N}^{}$}
\put(36,34){\scriptsize $e_{M,N}^B$}
\put(33,19){\scriptsize $e_{C,M\otimesB N}^{}$}
}
  }
\ep
}=\qquad
\raisebox{-85pt}{
\bp(100,170)
  \put(0,10){
\bild{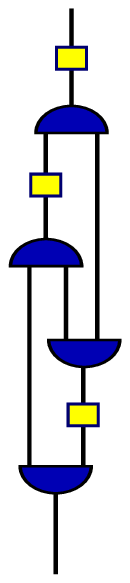}
\put(-40,-10){
\put(-5,1){\scriptsize $C\otimesA(M\otimesB N)$}
\put(12,181){\scriptsize $M\otimesB N$}
\put(1,122){\scriptsize $\rho_C^{}$}
\put(32,157){\scriptsize $r_{M,N}^B$}
\put(41,72){\scriptsize $e_{M,N}^{}$}
\put(-11,104){\scriptsize $r_{\!C,M}^{}$}
\put(37,141){\scriptsize $r_{M,N}^{}$}
\put(35.3,54){\scriptsize $e_{M,N}^B$}
\put(32,34){\scriptsize $e_{C,M\otimesB N}^{}$}
}
  }
\ep }
  \ee
Compose this morphism from the right with $r_{C,M\otimes_BN}$ and drop the resulting 
idempotent to get the transported left action of $\dot C$. Now composing with $f$ from 
the left and replacing $f\cir r_{M,N}^B\cir r_{M,N}$ by $r_{\dot M,\dot N}^{\dot B}$ 
and $e_{M,N}\cir e_{M,N}^B$ by $e_{\dot M,\dot N}^{\dot B}\cir f$ shows that $f$ 
also intertwines the left actions of $\dot C$. Similarly one shows that $f$ is 
also an isomorphism of $\dot D$-right modules.
\end{proof}

\begin{cor}\label{Moritatransport}
Assume that $B$ and $C$ are symmetric special Frobenius algebras in $\,\CAA$, and 
that $\MK BCP{P^\prime}$ is a Morita context in $\CAA$. Then 
$\MK{\dot B}{\dot C}{\dot P}{\dot P^\prime}$ is a Morita context in $\,\C$.
\end{cor}

\begin{proof}
Follows from lemma \ref{tensortransport} above. The commutativity of the diagrams 
\erf{moritadiagramme} in $\C$ follows from the commutativity of their counterparts in $\CAA$.
\end{proof}

We are now in a position to generalise the results of the previous section to the case
of algebras in arbitrary Morita classes. 

\begin{prop}
Let $H\,{\le}\,\Pic(\CAA)$ be a finite subgroup of the Picard group of 
$\,\CAA$ and assume that $\,\dim_A(\bigoplus_{h\in H}L_h) \,{\neq}\, 0$ for 
representatives $L_h$\ of $H$\ in $\CAA$. Then there exists an algebra 
$A^\prime$ in $\C$ and a Morita context $\MK{A}{A^\prime}{P}{P^\prime}$ in $\C$ 
such that $\Pi_{P^\prime,P}$ maps $H$ into the image of $\Psi_{\!A^\prime}$ in 
$\Pic(\C_{A^\prime|A^\prime})$. In other words, for any $h\iN H$ there is an 
algebra automorphism $\beta_h$ of $A^\prime$ such that the twisted bimodule  
${}_{\id} {A^\prime}_{\!\beta_h}$ is isomorphic to $(P^\prime\otA L_h)\otA P$.
\end{prop}

\begin{proof}
Applying proposition \ref{mengenschnitt} to the tensor unit of $\D$
yields, by the equivalence $\D\,{\simeq}\, \CAA$, a symmetric special Frobenius 
algebra $B$ in $\CAA$ and a Morita context $\MK ABQ{Q^\prime}$ in $\CAA$, such 
that there are automorphisms $\beta_h$ of $B$ and $B$-bimodule isomorphisms 
$F_h{:}~(Q^\prime\otA L_h)\otA Q\longcongTo {}_{\id}B_{\beta_h}$ for all 
$h\iN H$. By proposition \ref{algebratransport} and corollary \ref{Moritatransport} 
this gives rise to a Morita context $\MK{\dot A}{\dot B}{P}{P^\prime}$ in $\C$, 
where $P\eq\dot Q$ and $P^\prime\eq\dot Q^\prime$. The morphisms $F_h$ 
remain isomorphisms of bimodules when transported to $\C$, see proposition 
\ref{algebratransport}.
\\[3pt]
It follows that $(P^\prime\otimesA L_h)\otimesA\,P\cong {}_{\id} \dot B_{\beta_h}$ 
as $\dot B$-bimodules in $\C$.
\end{proof}

Since the algebra $\dot B$ might have a larger automorphism group in $\C$, its 
image under $\Psi_{\dot B}$ might be larger in $\Pic(\C_{\dot B|\dot B})$. So we 
cannot conclude that $H\,{\cong}\, \im(\Psi_{\dot B})$ in this case. But as we have 
$\Aut_{A|A}(B)\,{\le}\,\Aut_\C(\dot B)$ as subgroups, we can still generalise 
theorem \ref{gruppenschnitt}. This furnishes the main result of this paper:

\begin{thm}\label{thm:main}
Let $\C$ be a skeletally small sovereign abelian monoidal category with simple 
and absolutely simple tensor unit that is enriched over $\vect_\Bbbk$, with $\Bbbk$
a field of characteristic zero. Let $A$ be a simple and absolutely simple
symmetric special Frobenius algebra in $\C$, and let $\psi$ be a normalised 
three-co\-cycle describing the associator of the Picard category of $\,\CAA$. Let 
$H$ be an admissible subgroup of $\,\Pic(\CAA)$ (cf.\ definition \ref{admissible}).
\\[3pt]
Then there exist a symmetric special Frobenius algebra $A^\prime$ in $\C$ and 
a Morita context $\MK A{A^\prime}{P}{P^\prime}$ such that for each trivialisation 
$\omega$ of $\psi$ on $H$ (cf.\ definition \ref{trivialisation}) the following holds.
\def\leftmargini{1.8em}
\begin{itemize}
\item[\rm (i)]
There is an injective homomorphism $\alpha_\omega{:}~ H\To\Aut(A^\prime)$ such 
that $\,\Pi_{P,P^\prime}\cir\Psi_{\!A^\prime}\cir \alpha_\omega \eq \id_H$.
The assignment $\,\omega\,{\mapsto}\,\alpha_\omega$ is injective.
\item[\rm (ii)]
The fixed algebra of $\,\im(\alpha_\omega) \,\,{\le}\, \Aut(A^\prime)$ is isomorphic 
to $\dot Q(H,\omega)$, where $\dot Q(H,\omega)$ is the algebra $Q(H,\omega)$ in $\CAA$ 
as described in proposition \ref{trivialisierungsalgebra}, transported to $\C$.
\end{itemize}
\end{thm}

\begin{proof}
(i)\,
Denote again by $\D$ the strictification of the bimodule category $\CAA$. By 
propositions \ref{mengenschnitt} and \ref{trivialisierungsalgebra} and theorem 
\ref{gruppenschnitt} we find a symmetric special Frobenius algebra $B$ in 
$\CAA$ and a Morita context $\MK ABQ{Q^\prime}$ in $\CAA$ such that there is a 
homomorphism $\phi_\omega{:}~H\To \Aut_{A|A}(B)$ with $\Pi_{Q,Q^\prime}\cir\Psi_{\!B}{:} 
\Aut_{A|A}(B)\To \Pic(\CAA)$ as one-sided inverse. 
According to proposition \ref{algebratransport}, $\dot B$ is a symmetric special 
Frobenius algebra in $\C$ and $\MK{A}{\dot B}{P}{P^\prime}$ is a Morita context 
in $\C$ with $P\eq\dot Q$ and $P^\prime\eq\dot Q^\prime$. Further, we can extend 
$\Psi_{\!B}$ to $\Aut_\C(\dot B)$ by putting $\Psi_{\!\dot B}(\gamma) 
\eq [{}_{\id}\dot B_\gamma]$ for $\gamma\iN\Aut_\C(\dot B)$. Since we have 
$\Aut_{A|A}(B)\,{\subseteq}\, \Aut_\C(\dot B)$ as a subgroup, $\phi_\omega$ then
gives a homomorphism $\alpha_\omega{:}~ H\To\Aut_\C(\dot B)$ that has 
$\Pi_{P,P^\prime}\cir\Psi_{\!\dot B}$ as one-sided inverse. As was seen in theorem 
\ref{gruppenschnitt}, the assignment $\omega\,{\mapsto}\,\phi_\omega$ is a bijection, 
and hence the assignment $\omega\,{\mapsto}\,\alpha_\omega$ is still injective.
\\[3pt]
(ii)\,
Put $\beta_h \,{:=}\, \alpha_\omega(h)$. From theorem \ref{fixalgebrabestimmung} 
we know that the algebra $Q(H,\omega)$ is isomorphic to the fixed algebra $B^H$ 
in $\CAA$. It comes together with
an algebra morphism $i{:}~Q(H,\omega)\To B$ and a coalgebra morphism
$s{:}~B\To Q(H,\omega)$ such that $s\cir i \eq \id_{Q(H,\omega)}$ and 
$i\cir s \eq {|H|}^{-1}\sum_{h\in H}\beta_h$. Now let $f\iN\Hom_\C(X,\dot B)$ be a 
morphism obeying $\beta_h\cir f \eq f$ for all $h\iN H$. Then for 
$\bar f \,{:=}\, s\cir f{:}~X\To \dot Q(H,\omega)$ we find 
$i\cir\bar f \eq i\cir s\cir f \eq \frac{1}{|H|}\sum_{h\in H}\beta_h\cir f \eq f$, 
i.e.\ $\dot Q(H,\omega)$ satisfies the universal property of the fixed algebra in 
$\C$. So ${\dot B}^H\,{\cong}\, \dot Q(H,\omega)$ as objects in $\C$. By proposition 
\ref{algebratransport}, $i$ is still a morphism of Frobenius algebras when 
transported to $\C$. It follows that $\dot B^H\cong \dot Q(H,\omega)$ as
Frobenius algebras in $\C$.
\end{proof}

\newpage

\appendix

\section{Appendix} \label{appA}

Here we collect some facts about the category of bimodules over an algebra
in a monoidal category. Let $(\C, \oti, \Eins)$ be a an abelian sovereign 
strict monoidal category, enriched over $\vectk$ with $\Bbbk$ a
field of characteristic zero, and with simple and absolutely simple
tensor unit. Let $A$ be an algebra in $\C$.

The tensor product $X\otA Y$ of two $A$-bimodules
$X \,{\equiv}\, (X,\rho_X,\varrho_X)$ and $Y \,{\equiv}\, (Y,\rho_Y,\varrho_Y)$ 
is defined to be the cokernel of the morphism
$(\varrho_X \oti \id_Y - \id_X \oti \rho_Y) \iN\Hom(X\oti A\oti Y,X\oti Y)$.
In the following we will only deal with tensor products over symmetric special 
Frobenius algebras. In this case the notion of tensor product can equivalently 
be described as follows. Let $(A,m,\eta,\Delta,\epsilon)$ be a symmetric 
normalised special Frobenius algebra. Consider the morphism
  \be
P_{X,Y} ~:=~~
\raisebox{-21pt}{
\bp(50,57)
\bild{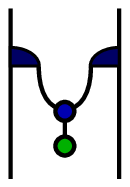}
\put(-39,-10){
\put(36,64){\scriptsize $Y$}
\put(34.5,1){\scriptsize $Y$}
\put(2.5,64){\scriptsize $X$}
\put(2,1){\scriptsize $X$}
}
\ep
}\in~\Hom_{A|A}(X\oti Y, X\oti Y) \,.
  \ee
~\\
Since $A$ is symmetric special Frobenius, $P_{X,Y}$ is an idempotent.
Writing $P_{X,Y} \eq e_{X,Y}\cir r_{X,Y}$ as a composition of a monic 
$e_{X,Y}$ and an epi $r_{X,Y}$, one can check that the morphism
$r_{X,Y}$ satisfies the universal property of the cokernel of 
$(\varrho_X \oti \id_Y - \id_X \oti \rho_Y)$. 

The object $X\otA Y$ is in the bimodule category $\CAA$ again. 
Indeed, right and left actions of $A$ on $X\otA Y$ can be defined by
  \be
\raisebox{-32pt}{
\bp(67,78)
  \put(0,11){
\bild{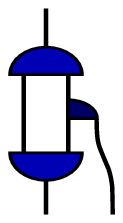}
\put(-40,-10){
\put(9,76){\scriptsize $X\otimesA Y$}
\put(6,2){\scriptsize $X\otimesA Y$}
\put(35,2){\scriptsize $A$}
\put(-10,56){\scriptsize $r_{X,Y}^{}$}
\put(-11,23){\scriptsize $e_{X,Y}^{}$}
}
  }
\ep
} \text{and}\qquad\quad
\raisebox{-32pt}{
\bp(40,65)
  \put(0,11){
\bild{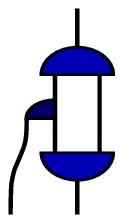}
\put(-40,-10){
\put(17,76){\scriptsize $X\otimesA Y$}
\put(18,2){\scriptsize $X\otimesA Y$}
\put(5,2){\scriptsize $A$}
\put(41,55){\scriptsize $r_{X,Y}^{}$}
\put(41,23){\scriptsize $e_{X,Y}^{}$}
}
  }
\ep
}
  \ee
The tensor product over $A$ of two morphisms
$f\iN \Hom_{A|A}(X,X^\prime)$ and $g\iN\Hom_{A|A}(Y,Y^\prime)$ is defined as
  \be
  f\otA g := r_{X^\prime,Y^\prime}\cir(f\oti g)\cir e_{X,Y} \,.
  \ee
So in particular the morphisms $e_{X,Y}$, $r_{X,Y}$ and $f\otA g$ are morphisms 
of $A$-bimodules. 

For any three $A$-bimodules $X$, $Y$, $Z$ one defines
  \be\label{bimod-associator}
\alpha_{X,Y,Z} ~:=~~\qquad
\raisebox{-57pt}{
\bp(70,122)
  \put(0,12){
\bild{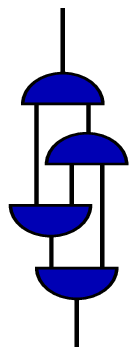}
\put(-40,-10){
\put(-3,114){\scriptsize $X\otimesA (Y\otimesA Z)$}
\put(-1,1){\scriptsize $(X\otimesA Y)\otimesA Z$}
\put(34,83){\scriptsize $r_{\!X,Y\otimesA Z}^{}$}
\put(-15,46){\scriptsize $e_{X,Y}^{}$}
\put(41,65){\scriptsize $r_{\!Y,Z}^{}$}
\put(-23,26){\scriptsize $e_{X\otimesA Z,Y}^{}$}
}
  }
\ep
}\in~\Hom_{A|A}\!\big((X\otA Y)\otA Z, X\otA(Y\otA Z)\big) \,,
  \ee
which is a morphism of $A$-bimodules.
These morphisms are in fact isomorphisms and have the following properties:
  \be\label{assoziatoreigenschaft}
\raisebox{-37pt}{
\bp(84,81)
  \put(0,7){
\bild{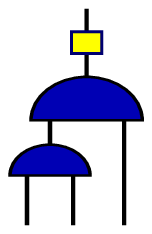}
\put(-40,-10){
\put(-3,78){\scriptsize $X\otimesA (Y\otimesA Z)$}
\put(0,2){\scriptsize $X$}
\put(14.7,2){\scriptsize $Y$}
\put(29,2){\scriptsize $Z$}
\put(40,46){\scriptsize $r_{\!X\otimesA Y, Z}^{}$}
\put(29,61.3){\scriptsize $\alpha_{X,Y,Z}^{}$}
\put(-18.6,28){\scriptsize $r_{\!X,Y}^{}$}
}
  }
\ep
}=~
\raisebox{-37pt}{
\bp(90,55)
  \put(0,7){
\bild{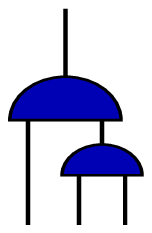}
\put(-40,-10){
\put(-7,78){\scriptsize $X\otimesA (Y\otimesA Z)$}
\put(2,2){\scriptsize $X$}
\put(18,2){\scriptsize $Y$}
\put(31,2){\scriptsize $Z$}
\put(35,45){\scriptsize $r_{\!X,Y\otimesA Z}^{}$}
\put(40.6,28){\scriptsize $r_{Y,Z}^{}$}
}
  }
\ep
} \text{and}\qquad
\raisebox{-37pt}{
\bp(83,55)
  \put(0,7){
\bild{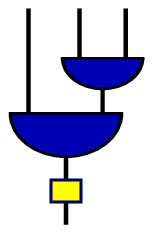}
\put(-40,-10){
\put(-9,2){\scriptsize $(X\otimesA Y)\otimesA Z$}
\put(1,78){\scriptsize $X$}
\put(17,78){\scriptsize $Y$}
\put(32,78){\scriptsize $Z$}
\put(33,34){\scriptsize $e_{X,Y\otimesA Z}^{}$}
\put(23,18.3){\scriptsize $\alpha_{X,Y,Z}^{}$}
\put(40,52){\scriptsize $e_{Y,Z}^{}$}
}
  }
\ep
}=\qquad
\raisebox{-37pt}{
\bp(80,55)
  \put(0,7){
\bild{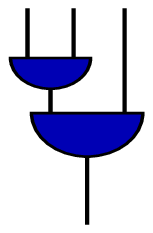}
\put(-40,-10){
\put(-5,2){\scriptsize $(X\otimesA Y)\otimesA Z$}
\put(1,78){\scriptsize $X$}
\put(17,78){\scriptsize $Y$}
\put(30,78){\scriptsize $Z$}
\put(40,35){\scriptsize $e_{X\otimesA Y,Z}^{}$}
\put(-19,53){\scriptsize $e_{X,Y}^{}$}
}
  }
\ep }
  \ee
This can be seen by writing out $\alpha_{X,Y,Z}$ and letting the 
occurring idempotents disappear using the properties of $A$ as a 
symmetric special Frobenius algebra.
  
An easy, albeit lengthy, calculation then shows that the isomorphisms 
$\alpha_{X,Y,Z}$ obey the pentagon condition for the associativity 
constraints in a monoidal category, i.e.\ one has
  \be
  (\id_U\OtA\,\alpha_{V,W,X}^{})\cir \alpha_{U,V\otimesA W,X}^{}
  \cir (\alpha_{U,V,W}^{}\,\OtA\id_X )
  = \alpha_{U,V,W\otimesA X}^{}\cir\alpha_{U\otimesA V, W,X}^{}
  \ee
for any quadruple $U$, $V$, $W$, $X$ of $A$-bimodules.
For an $A$-bi\-module $M$, unit constraints are given by
  \be \label{Bild38}
\rho^A(M)~=~~
\raisebox{-40pt}{
\bp(70,73)
  \put(0,8){
\bild{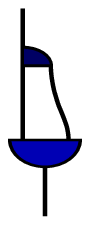}
\put(-40,-10){
\put(18.7,76){\scriptsize $M$}
\put(15,2){\scriptsize $M\otimesA A$}
\put(39.5,26){\scriptsize $e_{M,A}^{}$}
}
  }
\ep
} \text{and}\qquad \lambda^A(M)~=~~
\raisebox{-40pt}{
\bp(40,65)
  \put(0,8){
\bild{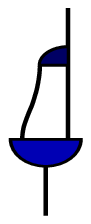}
\put(-40,-10){
\put(31,76){\scriptsize $M$}
\put(15,2){\scriptsize $A\otimesA M$}
\put(39.5,26){\scriptsize $e_{A,M}^{}$}
}
  }
\ep }
  \ee
with inverses
  \be\label{Bild39}
\rho^A(M)^{-1}~=~~
\raisebox{-40pt}{
\bp(70,95)
  \put(0,8){
\bild{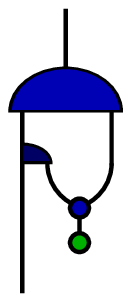}
\put(-40,-10){
\put(6,2){\scriptsize $M$}
\put(10,99){\scriptsize $M\otimesA A$}
\put(41,67){\scriptsize $r_{M,A}^{}$}
}
  }
\ep
} \text{and}\qquad \lambda^A(M)^{-1}~=~~
\raisebox{-40pt}{
\bp(50,98)
  \put(0,8){
\bild{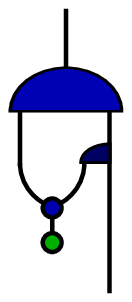}
\put(-40,-10){
\put(32,2){\scriptsize $M$}
\put(13,99){\scriptsize $A\otimesA M$}
\put(41.3,66){\scriptsize $r_{A,M}^{}$}
}
  }
\ep }
  \ee
This turns the category $\CAA$ into a (non-strict) monoidal category. 
We will now see that $\CAA$ is sovereign. Let $M$ be an $A$-bimodule and 
$M^\vee$ its dual as an object of $\C$. $M^\vee$ becomes an $A$-bimodule by
defining left and right actions of $A$ as
  \be
\raisebox{-32pt}{
\bp(29,73)
\bild{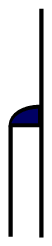}
\put(-40,-10){
\put(34,2){\scriptsize $M^\vee$}
\put(32,81){\scriptsize $M^\vee$}
\put(22,2){\scriptsize $A$}
}
\ep
}:=~~
\raisebox{-32pt}{
\bp(100,60)
\bild{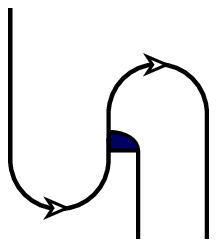}
\put(-40,-10){
\put(32,2){\scriptsize $M^\vee$}
\put(-23,81){\scriptsize $M^\vee$}
\put(14,2){\scriptsize $A$}
}
\ep
} \text{and}\qquad\qquad
\raisebox{-38pt}{
\bp(25,63)
  \put(0,6){
\bild{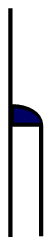}
\put(-40,-10){
\put(20,2){\scriptsize $M^\vee$}
\put(25,81){\scriptsize $M^\vee$}
\put(37,2){\scriptsize $A$}
}
  }
\ep
}:=~~
\raisebox{-38pt}{
\bp(70,60)
  \put(0,6){
\bild{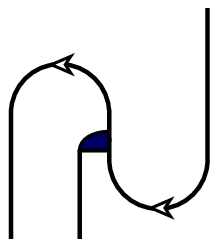}
\put(-40,-10){
\put(-25,2){\scriptsize $M^\vee$}
\put(33,81){\scriptsize $M^\vee$}
\put(-1,2){\scriptsize $A$}
}
  }
\ep }
  \ee
The structural morphisms of left and right dualities for $\CAA$ are given by
  \begin{eqnarray}
b_M^A~=\quad
\raisebox{-32pt}{
\bp(65,69)
\bild{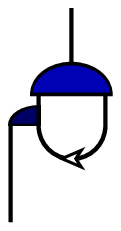}
\put(-40,-10){
\put(13,78){\scriptsize $M\otimesA M^\vee$}
\put(6,2){\scriptsize $A$}
\put(39,52){\scriptsize $r_{\!M,M^\vee}^{}$}
}
\ep
}, \qquad\quad~~
&&
d_M^A~=\qquad
\raisebox{-32pt}{
\bp(60,60)
\bild{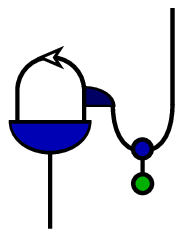}
\put(-40,-10){
\put(-14,2){\scriptsize $M^\vee\otimesA M$}
\put(-32,32){\scriptsize $e_{M^\vee,M}^{}$}
\put(36,78.5){\scriptsize $A$}
}
\ep
}
\nonumber\\
\tilde d_M^A~=\quad
\raisebox{-40pt}{
\bp(83,100)
  \put(0,8){
\bild{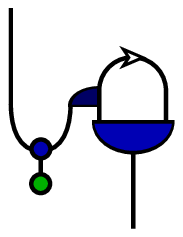}
\put(-40,-10){
\put(15,2){\scriptsize $M\otimesA M^\vee$}
\put(39,33){\scriptsize $e_{M,M^\vee}^{}$}
\put(-11,78.5){\scriptsize $A$}
  }
}
\ep
}, \qquad
&&
\tilde b_M^A~=\quad
\raisebox{-40pt}{
\bp(40,60)
  \put(0,8){
\bild{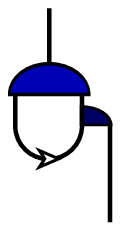}
\put(-40,-10){
\put(3,78){\scriptsize $M^\vee\otimesA M$}
\put(33,52){\scriptsize $r_{\!M^\vee,M}^{}$}
\put(34,2){\scriptsize $A$}
}
  }
\ep
}
\end{eqnarray}
One checks that this indeed furnishes dualities on $\CAA$, and that they 
coincide with those of $\C$ not only on objects, but also on morphisms. Thus 
if $\C$ is sovereign, then so $\CAA$. One also easily verifies that $\CAA$ is 
abelian and that its morphism groups are $\Bbbk$-vector spaces with 
$\Hom_{A|A}(M,N)\,{\subset}\,\Hom_\C(M,N)$ as subspaces, and hence 
$\dimf\Hom_{A|A}(M,N)\,{\le}\,\dimf\Hom_\C(M,N)$.
For $M$ an $A$-bimodule, we denote its left and right dimension as an object 
of $\CAA$ by ${\dimL}_A(M)$ and ${\dimR}_A(M)$, respectively. 
If $A$ is in addition an absolutely simple algebra, one has
  \be
  \dimR(M) = \dimR(A)\,{\dimR}_A(M)  \qquad{\rm and}\qquad
  \dimL(M) = \dimL(A)\,{\dimL}_A(M) \,.
  \ee

\vskip 5.5em


 \newcommand\wb{\,\linebreak[0]} \def\wB {$\,$\wb}
 \newcommand\Bi[2]    {\bibitem[#2]{#1}}
 \newcommand\JO[6]    {{\em #6}, {#1} {#2} ({#3}), {#4--#5} }
 \newcommand\J[7]     {{\em #7}, {#1} {#2} ({#3}), {#4--#5} {{\tt [#6]}}}
 \newcommand\BOOK[4]  {{\em #1\/} ({#2}, {#3} {#4})}

 \def\adma  {Adv.\wb in Math.}
 \def\coma  {Con\-temp.\wb Math.}
 \def\comp  {Com\-mun.\wb Math.\wb Phys.}
 \def\fic   {Fields\wb Inst.\wb Commun.}
 \def\joal  {J.\wB Al\-ge\-bra}
 \def\jpaa  {J.\wB Pure\wB Appl.\wb Alg.}
 \def\nupb  {Nucl.\wb Phys.\ B} 
 \def\phrl  {Phys.\wb Rev.\wb Lett.}
 \def\pajm  {Pa\-cific\wB J.\wb Math.} 
 \def\pmd   {Publ.\wb Math.\wb Debrecen}
 \def\slnm  {Sprin\-ger\wB Lecture\wB Notes\wB in\wB Mathematics}
 \def\tragr {Transform.\wb Groups}

\end{document}